\documentclass[11pt]{amsart}

\usepackage{amsfonts}
\usepackage{amsmath}
\usepackage{amssymb}
\usepackage{amsthm}
\usepackage{array}
\usepackage{eucal}
\usepackage[colorlinks,linkcolor=blue,urlcolor=cyan,citecolor=cyan,pagebackref]{hyperref}
\usepackage{lscape}
\usepackage{bbm}
\usepackage{mathdots}
\usepackage{mathrsfs}
\usepackage[all]{xypic}

\numberwithin{equation}{section}

\newtheorem{corollary}[equation]{Corollary}
\newtheorem{lem}[equation]{Lemma}
\newtheorem{proposition}[equation]{Proposition}
\newtheorem{theorem}[equation]{Theorem}
\newtheorem*{maintheorem}{Main Theorem}

\theoremstyle{definition}
\newtheorem*{definition}{Definition}

\theoremstyle{remark}

\newtheorem*{remark}{Remark}
\newtheorem{num}[equation]{}

\DeclareMathOperator{\D}{D} \DeclareMathOperator{\DIM}{dim}
\DeclareMathOperator{\GL}{GL} \DeclareMathOperator{\IM}{Im}
\DeclareMathOperator{\RK}{rk} \DeclareMathOperator{\SPEC}{Spec}
\DeclareMathOperator{\R}{H}

\newcommand{\cf}{{\em cf.}}
\newcommand{\ie}{{\em i.e.}}

\newcommand{\ra}{\rightarrow}
\newcommand{\lra}{\longrightarrow}
\newcommand{\hra}{\hookrightarrow}
\renewcommand{\(}{\left(}
\renewcommand{\)}{\right)}

\renewcommand{\b}[1]{\mathbf{#1}}
\renewcommand{\d}[1]{\mathbb{#1}}
\renewcommand{\c}[1]{\mathcal{#1}}
\newcommand{\f}[1]{\mathfrak{#1}}
\renewcommand{\r}[1]{\mathrm{#1}}
\newcommand{\s}[1]{\mathscr{#1}}
\renewcommand{\sf}[1]{\mathsf{#1}}

\newcommand{\GA}{\b{G}_{\r{a},k}}
\newcommand{\EL}{\overline{\d{Q}_{\ell}}}
\newcommand{\RG}{\r{R}\Gamma}
\newcommand{\RGC}{\r{R}\Gamma_{\r{c}}}

\newcommand{\Bun}{\r{Bun}}
\newcommand{\Coh}{\r{Coh}}
\newcommand{\Fl}{\r{Fl}}
\newcommand{\Mod}{\r{Mod}}
\newcommand{\M}{\c{M}}

\newcommand{\Qb}{{\overline{\c{Q}}}{}}

\newcommand{\Qm}{{_{\mu}\c{Q}}{}}
\newcommand{\Qmb}{{_{\mu}\overline{\c{Q}}}{}}

\newcommand{\Bunp}{{\r{Bun}'}{}}
\newcommand{\Cohp}{{\r{Coh}'}{}}
\newcommand{\Flp}{{\r{Fl}'}{}}
\newcommand{\Modp}{{\r{Mod}'}{}}
\newcommand{\Mp}{{\c{M}'}{}}
\newcommand{\Qp}{{\c{Q}'}{}}
\newcommand{\Qbp}{{\overline{\c{Q}}'}{}}
\newcommand{\Zm}{{_{\mu}\c{Z}}{}}

\newcommand{\Zp}{{\c{Z}'}{}}
\newcommand{\Ztp}{{\widetilde{\c{Z}}'}{}}

\newcommand{\Cc}{{\c{C}^{\circ}}{}}
\newcommand{\Y}{\c{Y}}
\renewcommand{\Yc}{{\c{Y}^{\circ}}{}}
\newcommand{\Ytc}{{\widetilde{\c{Y}}^{\circ}}{}}
\newcommand{\Yt}{{\widetilde{\c{Y}}}{}}

\newcommand{\AS}{\sf{AS}_{\psi}}
\newcommand{\ASt}{{\widetilde{\sf{AS}}}{}_{\psi}}
\newcommand{\ASb}{{\overline{\sf{AS}}}{}_{\psi}}
\newcommand{\Aut}{\sf{Aut}}
\newcommand{\T}{\sf{T}}
\newcommand{\W}{\sf{W}}
\newcommand{\Lau}{\sf{Lau}}
\newcommand{\Spr}{\sf{Spr}}

\begin{document}

\title{On quadratic distinction of automorphic sheaves}

\author{Yifeng Liu}

\address{Department of Mathematics, Columbia University, New York NY 10027}
\email{liuyf@math.columbia.edu}
\date{Jan 25, 2011}

\begin{abstract}
We prove a geometric version of a classical result on the
characterization of an irreducible cuspidal automorphic
representation of $\r{GL}_n(\d{A}_E)$ being the base change of a
stable cuspidal packet of the quasi-split unitary group associated
to the quadratic extension $E/F$, via the nonvanishing of certain
period integrals, called being distinguished. We show that certain
cohomology of an automorphic sheaf of $\r{GL}_{n,X'}$ is
nonvanishing if and only if the corresponding local system $E$ on
$X'$ is conjugate self-dual with respect to an \'{e}tale double
cover $X'/X$ of curves, which directly relates to the base change
from the associated unitary group. In particular, the geometric
setting makes sense for any base field.
\end{abstract}

\maketitle

\tableofcontents

\section*{Introduction}
\label{sec:intro}

\begin{num}
{\em Classical point of view: distinguished representations.} Let
$E/F$ be a quadratic extension of global fields with the Galois
group $\{1,\sigma\}$, $\Pi$ an irreducible cuspidal automorphic
representation of $\r{GL}_n(\d{A}_E)$. We consider the following two
period integrals.
 \begin{align*}
 \b{P}^+(\phi)&=\int_{\r{GL}_n(F)\backslash\r{GL}_n(\d{A}_F)^0}\phi(h)\r{d}h;\\
 \b{P}^-(\phi)&=\int_{\r{GL}_n(F)\backslash\r{GL}_n(\d{A}_F)^0}\phi(h)\omega_{E/F}(\det h)\r{d}h
 \end{align*}
where $\phi$ is a cusp form in the space of $\Pi$,
$\r{GL}_n(\d{A}_F)^0$ is the subgroup of $\r{GL}_n(\d{A}_F)$
consisting of elements whose determinant has norm $1$, and
$\omega_{E/F}$ is the quadratic character of $\d{A}_F^{\times}$
associated with $E/F$ via global class field theory. If $\b{P}^+$
(resp. $\b{P}^-$) is not identically zero on $\phi\in\Pi$, then we
say $\Pi$ is distinguished (resp. $\omega_{E/F}$-distinguished) by
$\r{GL}_{n,F}$, or simply distinguished (resp.
$\omega_{E/F}$-distinguished). Assuming the central character of
$\Pi$ is distinguished, it is proved in \cite{GJR} Proposition 1,
that $\Pi^{\vee}\cong\Pi^{\sigma}$ if and only if $\Pi$ is
distinguished (resp. distinguished or $\omega_{E/F}$-distinguished)
when $n$ is odd (resp. even). Here, $\Pi^{\vee}$ and $\Pi^{\sigma}$
denote the contragredient representation and the $\sigma$-twisted
representation, respectively. The proof is based on previous results
of Flicker \cite{Fl1}, \cite{Fl2}. From the point of view of
$L$-functions, $\Pi$ is distinguished if and only if the Asai
$L$-function $L(s,\Pi,\r{As})$ has a (simple) pole at $s=1$ and
$\Pi$ is $\omega_{E/F}$-distinguished if and only if
$L(s,\Pi\otimes\Omega,\r{As})$ has a (simple) pole at $s=1$, where
$\Omega$ is any automorphic character of $\d{A}_E^{\times}$ whose
restriction to $\d{A}_F^{\times}$ equals $\omega_{E/F}$
\cf~\cite{Fl1}.

The theory has a perfect geometric counterpart, in the framework of
Geometric Langlands Program. Instead of automorphic representations,
we work on automorphic sheaves and hence only consider the
everywhere unramified case. Then we will replace the period
integrals $\b{P}^+$ and $\b{P}^-$ by certain complexes of
$\ell$-adic sheaves on the base field, which now makes sense for any
field $k$. The isomorphism $\Pi^{\vee}\cong\Pi^{\sigma}$ has a
perfect Galois counterpart, namely the isomorphism $E^{\vee}\cong
E^{\sigma}$ for the corresponding local system $E$ of $\Pi$. The use
of poles of $L$-functions will implicitly appear in our proof as
cohomology. Now we are going to state our situation and Main Theorem
in more details.\\
\end{num}

\begin{num}
{\em Geometrization and Main Theorem.} Let $k$ be a field (which can
and will be assumed algebraically closed) and $\ell$ a prime number
invertible in $k$. Let $X$ be a connected smooth proper curve over
$k$ and $\mu:X'\ra X$ be a proper \'{e}tale morphism of degree $2$
with $X'$ connected. Let $\sigma$ be the unique nontrivial
isomorphism such that the following diagram commutes
 \begin{align*}
 \xymatrix{
  X' \ar[rr]^-{\sigma} \ar[dr]_-{\mu}
                &  &    X' \ar[dl]^-{\mu}    \\
                & X                 }
 \end{align*}

For a positive integer $n$, we denote $\Bun_n$ (resp. $\Bunp_n$) the
moduli stack of rank-$n$ vector bundles on $X$ (resp. $X'$). Then
the pullback of vector bundles under $\mu$ induces a morphism
$\mu_n:\Bun_n\ra\Bunp_n$. The stack $\Bun_n$ (resp. $\Bunp_n$) is a
disjoint union of its connected components $\Bun_n^d$ (resp.
$\Bunp_n^d$) parameterizing those bundles of normalized degree
(\cf~\ref{num:intro-notations}) $d$ for $d\in\d{Z}$, then we have
$\mu_n:\Bun_n^d\ra\Bunp_n^{2d}$. Moreover, we have an isomorphism
$\sigma_n:\Bunp_n\ra\Bunp_n$, induced by the pullback under
$\sigma$.

For a local system $E$ on $X'$ of rank $n$, we let $E^{\vee}$ be its
dual system and $E^{\sigma}=\sigma^*E$. If $E$ is irreducible, then
the geometric Langlands correspondence (proved by Drinfeld \cite{Dr}
for $n=2$, formulated by Laumon \cite{La1}, \cite{La2} and proved by
Frenkel-Gaitsgory-Vilonen for general $n$ \cite{FGV2}) associates
$E$ a $\EL$ perverse sheaf on $\Bunp_n$, denoted by $\Aut_E$, which
is irreducible and nontrivial on each $\Bunp_n^d$ and satisfies the
Hecke property with respect to $E$. If we denote $\EL$ the trivial
rank-one local system, then we have a canonical decomposition
$\mu_*\EL=\EL\oplus L_{\mu}$ for a degree 2 rank-one local system
$L_{\mu}$ on $X$. The determinant of vector bundles induces a
morphism $\r{det}:\Bun_n\ra\Bun_1=\r{Pic}_X$ preserving degree. For
a rank-one local system $L$ on $X$, we define $\T_L=\det^*\sf{A}_L$,
where $\sf{A}_L$ is the local system placed at degree zero shifted
from $\Aut_L$. For simplicity, we let $\T_{\mu}:=\T_{L_{\mu}}$.

We will define in Section \ref{sec:asai} one of the primary objects
in this paper, the Asai local system $\r{As}(E)$, which is a local
system of rank $n^2$ on $X$. It is constructed by the geometric
analogue of the classical method, called multiplicative induction or
twisted tensor product \cf~\cite{Pr} Section 7, for the Asai
representation of the Galois group. The following is our Main
Theorem.

\begin{maintheorem}\nonumber\label{the:intro-main}
Let notations be as above and $E$ an irreducible local system on
$X'$ of rank $n$, consider the following statements:\\
$(a)$. $E^{\vee}\simeq E^{\sigma}$;\\
$(b)$. $\D\Aut_E\simeq\sigma_n^*\Aut_E$, where $\D$ denotes the Verdier duality;\\
$(c)$. $\RGC(\Bunp_n^0,\Aut_E\otimes\sigma_n^*\Aut_E)\neq0$
\footnote[1]{One may complain that $\RGC$ is not defined by
Laszlo-Olsson \cite{LO} since $\Bunp_n^0$ is not of finite type in
general, but we only talk about the triviality of this complex which
is well-defined. In fact, even the complex makes sense since
actually the support of $\Aut_E$ on $\Bunp_n^0$ is of finite type.};\\
$(d^+)$. $\EL\subset\r{As}(E)$;\\
$(e^+)$. $\RGC(\Bun_n^0,\mu_n^*\Aut_E)\neq0$;\\
$(d^-)$. $L_{\mu}\subset\r{As}(E)$;\\
$(e^-)$. $\RGC(\Bun_n^0,\mu_n^*\Aut_E\otimes\T_{\mu})\neq0$.\\
Then, we have the following equivalence
\begin{align*}
(a)\Longleftrightarrow(b)\Longleftrightarrow(c)\Longleftrightarrow(d^+)\text{
or }(d^-)
\end{align*}
and
\begin{align*}
(d^+)\Longleftrightarrow(e^+);\qquad(d^-)\Longleftrightarrow(e^-).
\end{align*}\\
\end{maintheorem}

The theorem has the following corollaries on the direct image
complexes which are not easy to see directly.

\begin{corollary}\label{cor:intro}
(1) We have
 \begin{align*}
 \RGC(\Bun_n^0,\mu_n^*\Aut_E)=0 &\Longleftrightarrow
 \RG(\Bun_n^0,\mu_n^*\Aut_E)=0,\\
 \RGC(\Bun_n^0,\mu_n^*\Aut_E\otimes\T_{\mu})=0 &\Longleftrightarrow
 \RG(\Bun_n^0,\mu_n^*\Aut_E\otimes\T_{\mu})=0;
 \end{align*}\\
(2) Between $\RGC(\Bun_n^0,\mu_n^*\Aut_E)$ and
$\RGC(\Bun_n^0,\mu_n^*\Aut_E\otimes\T_{\mu})$, there is at most one
which is nontrivial and there is one, if and only if $E^{\vee}\simeq
E^{\sigma}$.\\
\end{corollary}

\end{num}

\begin{num}
{\em Mirabolic calculation.} The proof of Main Theorem relies on the
computation of the direct image complex of certain sheaves on the
moduli stack corresponding to the mirabolic subgroup of $\GL_n$. Let
$\M_n$ (resp. $\Mp_n$) be the moduli stack classifying pairs
$(\s{M},s_1)$ (resp. $(\s{M}',s_1')$), where $\s{M}$ (resp.
$\s{M}'$) is an object in $\Bun_n$ (resp. $\Bunp_n$) and
$s_1:\Omega_X^{n-1}\hra\s{M}$ (resp.
$s'_1:\Omega_{X'}^{n-1}\hra\s{M}'$) is an inclusion of
$\s{O}_X$-modules (resp. $\s{O}_{X'}$-modules). Here,
$\Omega_{\bullet}$ and $\s{O}_{\bullet}$ stand for the sheaf of
differentials and the structure sheaf, respectively and
$\Omega_{\bullet}^{n-1}$ means $\Omega_{\bullet}^{\otimes n-1}$.
Again, the pullback of pairs $(\s{M},s_1)$ under $\mu$ induces a
morphism $\ddot{\mu}_n:\M_n\ra\Mp_n$ and we have obvious morphism
$\pi:\M_n\ra\Bun_n$ (resp. $\pi':\Mp_n\ra\Bunp_n$) by forgetting
$s_1$ (resp. $s'_1$). Let $\M_n^d$ (resp. $\Mp_n^d$) be the inverse
image of $\Bun_n^d$ (resp. $\Bunp_n^d$) under $\pi$ (resp. $\pi'$).
Define $\W_E={\pi'}^*\Aut_E$ (up to a cohomological shift). We have
the following theorem.

\begin{theorem}\label{the:intro-mira}
Let notations be as above and $E$ an irreducible local system on
$X'$ of rank $n$. Let $L$ be a rank-one local system on $X$, then we
have for $d\geq0$ a canonical isomorphism
 \begin{align*}
 \RGC(\M_n^d,\ddot{\mu}_n^*\W_E\otimes\pi^*\T_L)\overset{\sim}{\lra}\RG(X^{(d)},(\r{As}(E)\otimes
 L)^{(d)})[2d]
 \end{align*}
where $X^{(d)}$ and $(\r{As}(E)\otimes L)^{(d)}$ denote the $d$-th
symmetric product of the curve $X$ and its local system
$\r{As}(E)\otimes L$, respectively.\\
\end{theorem}

The proof of the above theorem follows the same line in \cite{Ly}
for the geometrized Rankin-Selberg method due to Lysenko, which can
be viewed as the geometrization of the well-known classical way
treating the Rankin-Selberg integral of Jacquet, Piatetskii-Shapiro
and Shalika. We modify the argument in \cite{Ly} to our situation,
just as the modification of the classical argument in \cite{GJR},
\cite{Fl1}. It is interesting and useful to make these
methodological comparison, which provides certain hint for the proof
in the geometric counterpart.\\
\end{num}

\begin{num}
{\em Connection with functoriality.} Let us go back to the classical
situation. As pointed out in \cite{GJR}, \cite{Fl2}, the isomorphism
$\Pi^{\vee}\cong\Pi^{\sigma}$ has important meaning on the
functorial lifting. Let $\r{U}_{n,E/F}$ denote the quasi-split
unitary group of $n$ variables with respect to the quadratic
extension $E/F$. When $n$ is odd, the representation $\Pi$ with
central character being trivial on $\d{A}_F^{\times}$ which
satisfies $\Pi^{\vee}\cong\Pi^{\sigma}$ should be a (stable) base
change of a stable cuspidal packet of $\r{U}_{n,E/F}$. When $n$ is
even, the representation $\Pi$ satisfying
$\Pi^{\vee}\cong\Pi^{\sigma}$ should also be a base change of a
stable cuspidal packet of $\r{U}_{n,E/F}$, either an unstable base
change or a stable base change, according to whether $\Pi$ is
distinguished or $\omega_{E/F}$-distinguished.

In the geometric situation. Let $\r{U}_{n,X'/X}$ be the quasi-split
unitary group with respect to $X'/X$, which is a reductive group
over $X$ and whose Langlands dual group
$^L\r{U}_{n,X'/X}\simeq\r{GL}_n(\EL)\rtimes\pi_1(X)$ (the action
defining the semi-product factors through $\pi_1(X')$). Since we
have local systems parameterizing automorphic sheaves, it is easy to
see the following fact. When $n$ is odd, an irreducible local system
$E$ on $X'$ of rank $n$ is a (stable) base change (in fact,
restriction) of a $^L\r{U}_{n,X'/X}$-local system on $X$ if and only
if $E^{\vee}\simeq E^{\sigma}$ and $\r{As}(\wedge^nE)=\EL$
\cf~\ref{sec:asai}. When $n$ is even, an irreducible local system
$E$ on $X'$ of rank $n$ is an unstable or stable base change of a
$^L\r{U}_{n,X'/X}$-local system on $X$ if and only if
$E^{\vee}\simeq E^{\sigma}$. It is a dichotomy of being unstable or
stable which are corresponding to $(d^+)$ or $(d^-)$ in Main
Theorem. Hence, our Main Theorem provides a geometric/cohomological
characterization of automorphic sheaves of $\r{GL}_{n,X'}$ which
should correspond to the (conjectural) automorphic sheaves of
$\r{U}_{n,X'/X}$ via geometric Langlands
functoriality.\\
\end{num}

\begin{num}
{\em Structure of the paper.} In Section \ref{sec:asai}, we will
define the so-called Asai local system, which is the geometric
analogue of the classical Asai representation of the Galois group.
These local systems are the main objects studied in this paper. In
Section \ref{sec:steps}, we prove Main Theorem and its Corollary
assuming Theorem \ref{the:intro-mira}. The proof implicitly uses the
idea of poles which appear as cohomology on higher symmetric
products of the base curve. The rest sections are devoted to prove
Theorem \ref{the:intro-mira}.

In Section \ref{sec:res}, after recalling the definition of Laumon's
sheaf and the Whittaker sheaf, we reduce Theorem
\ref{the:intro-mira} to certain formula (see Proposition
\ref{pro:res-coh}) relating direct images of Whittaker sheaves and
symmetric products of constructible sheaves on the base curve.
Section \ref{sec:red} and Section \ref{sec:direct} are responsible
for the proof of this formula.\\
\end{num}

\begin{num}\label{num:intro-notations}
{\em Notations and Conventions.}

We fix a field $k$ (which can and will be assumed algebraically
closed) and a prime $\ell$ invertible in $k$. Fix an algebraic
closure $\EL$ of $\d{Q}_{\ell}$. We work with Artin stacks, \ie,
algebraic stacks in the smooth topology \cf~\cite{LMB}, locally of
finite type over $k$. In the main body of the paper, we will assume
$k$ has positive characteristic $p$ and work in the derived category
of unbounded complexes of $\EL$-sheaves with constructible
cohomology on an Artin stack $\c{X}$ locally of finite type over
$k$, in the sense of \cite{LO}, which we denote by $\D(\c{X})$. The
operations $f^*,f^!,f_*,f_!,\otimes$ are only applied to locally
bounded complexes and understood in the derived sense. For
$K\in\D(\c{X})$, we let $\R^iK$ be its $i$-th cohomology sheaf with
respect to the {\em usual} t-structure. In particular, we denote
$\EL$ the constant sheaf placed in degree 0. We fix a nontrivial
additive character $\psi:\d{F}_p\ra\EL^{\times}$ and denote $\AS$
the corresponding Artin-Schreier sheaf on the $k$-affine line $\GA$.
All the results and proof work perfectly in the case when $k$ has
characteristic 0, after replacing $\D(\c{X})$ by the derived
category of $\s{D}$-modules with holonomic cohomology and
the Artin-Schreier sheaf $\AS$ by the $\s{D}$-module $``e^x"$.\\

We say a stack $\c{X}$ classifies something, we always mean that it
is an Artin stack locally of finite type over $k$ such that for any
$k$-scheme $S$ and morphism $S\ra S'$, the groupoid
$\r{Hom}(S,\c{X})$ and the functor
$\r{Hom}(S',\c{X})\ra\r{Hom}(S,\c{X})$ are clearly understood. For
example, for the stack $\M_n$ defined above, objects in
$\r{Hom}(S,\M_n)$ are the pairs $(\s{M}_S,s_{1,S})$ where $\s{M}_S$
is a vector bundle on $S\times X$ of rank $n$ and
$s_{1,S}:\s{O}_S\boxtimes\Omega^{n-1}_X\hra\s{M}_S$ is an inclusion
of $\s{O}_{S\times X}$-modules such that the quotient $\s{M}_S/\IM
s_{1,S}$ is $S$-flat; morphisms between $(\s{M}_S,s_{1,S})$ and
$(\s{N}_S,t_{1,S})$ are isomorphisms
$f:\s{M}_S\overset{\sim}{\lra}\s{N}_S$,
such that $f\circ s_{1,S}=t_{1,S}$.\\

We identify the lattice of coweights of $\GL_m$: $\Lambda_m$ with
$\d{Z}^m$. Let
 \begin{align*}
 \Lambda_{m,\r{pos}}&=\{\lambda=(\lambda_1,...,\lambda_m)\in\Lambda_m\:|\:\lambda_1+\cdots+
 \lambda_i\geq0, i=1,...,m\};\\
 \Lambda_{m,\r{eff}}&=\{\lambda=(\lambda_1,...,\lambda_m)\in\Lambda_m\:|\:\lambda_i\geq0, i=1,...,m\}
 \subset\Lambda_{m,\r{pos}};\\
 \Lambda_{m,+}&=\{\lambda=(\lambda_1,...,\lambda_m)\in\Lambda_m\:|\:\lambda_1\geq\lambda_2\geq\cdots\geq
 \lambda_m\geq0\}\subset\Lambda_{m,\r{eff}};\\
 \Lambda_{m,-}&=\{\lambda=(\lambda_1,...,\lambda_m)\in\Lambda_m\:|\:0\leq\lambda_1\leq\lambda_2\leq\cdots\leq
 \lambda_m\}\subset\Lambda_{m,\r{eff}}.
 \end{align*}
For $d\in\d{Z}$, we let
$\Lambda_{m,?}^d=\{\lambda=(\lambda_1,...,\lambda_m)\in\Lambda_{m,?}\:|\:\lambda_1+\cdots+\lambda_m=d\}$
for $?=\emptyset,\r{pos},\r{eff},+,-$. Finally, we denote $\f{S}_m$
the group of $m$-permutations.\\

For a connected smooth scheme $S$ over $k$ and $d\geq0$, we define
its $d$-th symmetric product $S^{(d)}$ as the quotient of $S^d$
under the action by $\f{S}_d$ ($S^{(0)}=\SPEC k$). For a local
system $E$ on $S$ of rank $n$, we let
$E^{(d)}=\(\r{sym}_!E^{\boxtimes\:d}\)^{\f{S}_d}$ be the $d$-th
symmetric product of $E$, where $\r{sym}:S^d\ra S^{(d)}$ is the
natural projection. Hence $E^{(d)}[d\cdot\DIM S]$ is a perverse
sheaf on $S^{(d)}$. It is irreducible if and only if $E$ is. We
denote by $\det E=\wedge^mE$ the determinant of $E$, where $m$ is
rank of $E$.

We fix $X$ to be a connected smooth proper curve over $k$, then
$X^{(d)}$ is the variety of effective divisors of degree $d$ on $X$,
\ie, collections of $d$ closed points. For
$\lambda\in\Lambda_{m,\r{pos}}^d$, we let
$X^{\lambda}=\prod_{i=1}^mX^{(\lambda_1+\cdots+\lambda_i)}$. Let
$X^{\lambda}_+$ (resp. $X^{\lambda}_-$) be the closed subscheme of
$X^{\lambda}$ such that $(D_1,D_1+D_2,...,D_1+\cdots+D_m)\in
X^{\lambda}$ belongs to $X^{\lambda}_+$ (resp. $X^{\lambda}_-$) if
and only if $0\leq D_1\leq\cdots\leq D_m$ (resp. $D_1\geq\cdots\geq
D_m\geq0$), which is nonempty if and only if
$\lambda\in\Lambda_{m,+}^d$ (resp. $\lambda\in\Lambda_{m,-}^d$).
They are both contained in the subscheme
$X^{\lambda}_{\r{eff}}:=\prod_{i=1}^mX^{(\lambda_i)}\hra
X^{\lambda}$ (when $\lambda_i\geq0$). We denote by
$\r{sym}^{\lambda}:X^{\lambda}\ra X^{(d)}$ the projection to the
last factor, and also for its restriction to
$X^{\lambda}_{\r{eff}}$, $X^{\lambda}_+$ or $X^{\lambda}_-$.

Let $X$ be as above and $\s{M}$ a coherent sheaf on $X$, its degree
$\deg\s{M}$ is normalized to be its usual degree minus
$m(m-1)(g-1)$, where $m=\RK\s{M}$ is the rank of $\s{M}$ and $g$ is
the genus of $X$. Hence $\deg\bigoplus_{i=0}^{m-1}\Omega_X^i=0$.\\
\end{num}

\section{Asai local systems and conjugate self-duality}
\label{sec:asai}

\begin{num}
{\em Asai local systems.} Consider an \'{e}tale morphism $\mu:X'\ra
X$ of degree $l\geq1$ with $X'$ being connected, and a local system
$E$ of rank $m\geq1$ on $X'$. Since $\mu$ is \'{e}tale, we have a
canonical isomorphism
 \begin{align*}
 X\overset{\sim}{\lra}\underset{l\text{
 times}}{\(\underbrace{X'\times_X\cdots\times_XX'}-\Delta\)}/\f{S}_l.
 \end{align*}
Composing with the closed embedding
$X'\times_X\cdots\times_XX'\hra{X'}^l$, we get a morphism
 \begin{align*}
 \mu^{(1)}:X\lra{X'}^{(l)}-\Delta
 \end{align*}
where we abuse the notation $\Delta$ for various kinds of diagonals.

\begin{definition}
We define the {\em Asai local system} $\r{As}(E)$ by
 \begin{align*}
 \r{As}(E):=\mu^{(1)*}\(\r{sym}_!E^{\boxtimes\:l}\)^{\f{S}_l}=\mu^{(1)*}\(E^{(l)}\).
 \end{align*}
Since the symmetrization morphism $\r{sym}:{X'}^l\ra{X'}^{(l)}$ is
\'{e}tale with Galois group $\f{S}_l$ away from $\Delta$,
$\r{As}(E)$ is a local system on $X$ of rank $m^l$.\\
\end{definition}

\begin{remark}
Surprisingly, the construction of Asai local systems is canonical,
not like the construction of the classical Asai representation. For
the later one, we need to choose a set of representatives for the
left coset of $\pi_1(X',x')$ in $\pi_1(X,x)$ and show that the
isomorphism class of representation obtained from the process of
multiplicative induction is independent of the choice \cf~\cite{Pr}
Section 7.
\end{remark}

From now on, we will only consider the case $l=2$. Let $\sigma:X'\ra
X'$ be the unique nontrivial isomorphism such that
$\mu=\mu\circ\sigma$, and $E^{\sigma}=\sigma^*E$. Since $l=2$, we
have $\mu_*\EL=\EL\oplus L_{\mu}$ for some rank-one local system
$L_{\mu}$ on $X$.

\begin{lem}
(1) We have $\r{As}(E^{\vee})\simeq\r{As}(E)^{\vee}$ and $\r{As}(E)\simeq\r{As}(E^{\sigma})$.\\
(2) There is a canonical isomorphism $\mu_*(E\otimes
E^{\sigma})\overset{\sim}\lra\r{As}(E)\oplus\r{As}(E)\otimes
L_{\mu}$.
\end{lem}
\begin{proof}
(1) is straightforward. For (2), consider the following diagram
 \begin{align*}
 \xymatrix{
   X' \ar[d]_-{\mu} \ar[rr]^-{(1,\sigma)} && X'\times X'-\Delta \ar[d]^-{\r{sym}} \\
   X \ar[rr]^-{\mu^{(1)}} && {X'}^{(2)}-\Delta   }
 \end{align*}
which is Cartesian. Over $X'\times X'-\Delta$, we have a canonical
isomorphism $\r{sym}^*E^{(2)}\overset{\sim}{\lra}E^{\boxtimes\:2}$.
Hence
 \begin{align*}
 \mu^*\r{As}(E)\simeq\mu^*\mu^{(1)*}E^{(2)}\simeq(1,\sigma)^*\r{sym}^*E^{(2)}\overset{\sim}{\lra}
 (1,\sigma)^*E^{\boxtimes\:2}\simeq E\otimes E^{\sigma}
 \end{align*}
By the projection formula, we have
 \begin{align*}
 \mu_*(E\otimes
 E^{\sigma})\simeq\r{As}(E)\otimes\mu_*\EL\simeq\r{As}(E)\oplus\r{As}(E)\otimes
 L_{\mu}.
 \end{align*}\\
\end{proof}
\end{num}

\begin{num}
{\em Conjugate self-duality.} We make the following definition as
the geometric analogue of the classical representation theory
\cf~\cite{GGP}.

\begin{definition}
Let the situation be as above, we say $E$ is {\em conjugate
self-dual} if there is a nontrivial map $b:E\otimes
E^{\sigma}\ra\EL$ between local systems, \ie, $E^{\vee}\simeq
E^{\sigma}$.

Let $b^{\sigma}:E\otimes E^{\sigma}\simeq E^{\sigma}\otimes
E=\sigma^*(E\otimes
E^{\sigma})\overset{\sigma^*b}{\lra}\sigma^*\EL=\EL$ be the
composition of $\sigma^*b$ with the transposition of two factors.
Then we say $E$ is {\em conjugate orthogonal} (resp. {\em conjugate
symplectic}) if moreover, we have $b^{\sigma}=b$ (resp.
$b^{\sigma}=-b$) and denote $c(E)=1$ (resp. $c(E)=-1$).\\
\end{definition}

It is easy to see that if $E$ is irreducible and conjugate
self-dual, then it is either conjugate orthogonal or conjugate
symplectic, and $c(E)=c(E^{\vee})=c(E^{\sigma})$. Its determinant
$\det E$ is also conjugate self-dual and $c(\det E)=c(E)^m$ where
$m$ is the rank of $E$. We have the following criterion.

\begin{lem}\label{lem:asai-conjugate}
For irreducible conjugate self-dual local system $E$, $c(E)=1$
(resp. $c(E)=-1$) if and only if $\r{As}(E)$ contains $\EL$ (resp.
$L_{\mu}$).
\end{lem}
\begin{proof}
It is obvious that we have dichotomy on both sides, hence we only
need to show, for example, that $\EL\subset\r{As}(E)$ implies
$c(E)=1$. By duality and adjunction, we have
 \begin{align*}
 \(\r{sym}_!E^{\vee\boxtimes\:2}\)^{\f{S}_2}=E^{\vee(2)}\lra\mu^{(1)}_{\;*}\EL
 \end{align*}
on ${X'}^{(2)}-\Delta$. By the pullback under $\r{sym}$, we get a
nontrivial map
$b':E^{\vee\boxtimes\:2}\simeq\r{sym}^*E^{\vee(2)}\lra\r{sym}^*\mu^{(1)}_{\;*}\EL$
which makes the following diagram commutative
 \begin{align*}
 \xymatrix{
  E^{\vee}\boxtimes E^{\vee} \ar[rr]^-{\sim} \ar[dr]_-{b'}
                &  &    E^{\vee}\boxtimes E^{\vee} \ar[dl]^-{b'}    \\
                &    \r{sym}^*\mu^{(1)}_{\;*}\EL              }
 \end{align*}
where the upper arrow is the transposition. By the pullback under
$(1,\sigma)$, we get $b^{\vee}:E^{\vee}\otimes E^{\vee\sigma}\ra\EL$
such that $c(E^{\vee})=1$, where $b^{\vee}=(1,\sigma)^*b'$. But this
implies that $c(E)=1$.\\
\end{proof}
\end{num}

\begin{num}\label{num:asai-sym}
{\em symmetrization.} For $d\geq0$, we identify $\f{S}_d$ with the
set of $d$-by-$d$ permutation matrix. Let $\f{T}_{2d}$ be the subset
of $\f{S}_{2d}$ consisting of those symmetric permutation matrix
whose diagonal entries are $0$. For $t=(t_{ij})\in\f{T}_{2d}$, we
define $1=i_1(t)<\cdots<i_d(t)$ and $i_a(t)<j_a(t)$ ($a=1,...,d$) by
the condition that $t_{i_a(t)j_a(t)}=1$. Then they are uniquely
determined and $\{i_a(t),j_a(t)\:|\:a=1,...,d\}=\{1,...,2d\}$. Let
$\mu^{(d)}:X^{(d)}\ra\({X}'^{(2)}\)^{(d)}\ra{X'}^{(2d)}$ be the
composition map. For any $t=(t_{ij})\in\f{T}_{2d}$, we define a
morphism:
 \begin{align*}
 \f{n}_t:{X'}^d\lra{X'}^{2d}\times_{{X'}^{(2d)}}X^{(d)}
 \end{align*}
such that the $i_a(t)$-th component of $\f{n}_t(x_1,...,x_d)$ is
$x_a$ and the $j_a(t)$-th component of $\f{n}_t(x_1,...,x_d)$ is
$\sigma(x_a)$. It is obvious that $\f{n}_t(x_1,...,x_d)$ locates in
${X'}^{2d}\times_{{X'}^{(2d)}}X^{(d)}$. Define
 \begin{align*}
 \f{n}:=\bigsqcup_{\f{T}_{2d}}\f{n}_t:\b{X}_d:=\bigsqcup_{\f{T}_{2d}}{X'}^d_t\lra{X'}^{2d}\times_{{X'}^{(2d)}}X^{(d)}
 \end{align*}
as the disjoint union over all $t\in\f{T}_{2d}$ with
${X'}^d_t={X'}^d$. Moreover, let $\r{pr}$ denotes the projection
${X'}^{2d}\times_{{X'}^{(2d)}}X^{(d)}\ra X^{(d)}$. Then we have

\begin{lem}\label{lem:asai-sym}
(1) The scheme ${X'}^{2d}\times_{{X'}^{(2d)}}X^{(d)}$ is of pure
dimension $d$ and its irreducible components one-to-one correspond
to schemes ${X'}^d_t$ for $t\in\f{T}_{2d}$;\\
(2) The morphism $\f{n}$ is finite and isomorphic over an open dense
subscheme of ${X'}^{2d}\times_{{X'}^{(2d)}}X^{(d)}$. Since $\b{X}_d$
is smooth, we have $\f{n}_!\EL[d]\simeq\r{IC}$ where $\r{IC}$
denotes the intersection cohomology sheaf on
${X'}^{2d}\times_{{X'}^{(2d)}}X^{(d)}$;\\
(3) We have a canonical isomorphism
$\r{pr}_!\(\(E^{\boxtimes\:2d}\boxtimes\EL\)\otimes\f{n}_!\EL\)^{\f{S}_{2d}}\overset{\sim}{\lra}\r{As}(E)^{(d)}$
on $X^{(d)}$, where $\f{S}_{2d}$ acts on ${X'}^{2d}$ naturally by
permuting $2d$ factors.
\end{lem}
\begin{proof}
For any $t\in\f{T}_{2d}$, we define $\b{X}_t$ to be the subscheme of
${X'}^{2d}\times_{{X'}^{(2d)}}X^{(d)}$ consisting of points
$(x_1,...,x_{2d})$ such that $x_i=\sigma(x_j)$ if and only if
$t_{ij}=1$. Then $\b{X}_t$ forms a stratification of
${X'}^{2d}\times_{{X'}^{(2d)}}X^{(d)}$. Moreover, define
$\f{n}^t:\b{X}_t\ra{X'}^d_t$ by
$(x_1,...,x_{2d})\mapsto(x_{i_1(t)},...,x_{i_d(t)})$. Then
$\f{n}_t\circ\f{n}^t=\r{id}$. Hence $\f{n}_t$ is an isomorphism onto
$\b{X}_t$ away from the diagonal and
${X'}^{2d}\times_{{X'}^{(2d)}}X^{(d)}$ is of pure dimension $d$.
Since $\f{n}_t$ is a closed immersion for each $t$ and ${X'}^d_t$ is
smooth, (1) and (2) follow. The
statement (3) is straightforward once knowing (1) and (2).\\
\end{proof}
\end{num}

\section{Steps for the proof}
\label{sec:steps}

\begin{num}\label{num:steps-1}
{\em Proof of $(d^{\pm})\Longleftrightarrow(e^{\pm})$.} We first
prove the equivalence between $(d^?)$ and $(e^?)$ for $?=+,-$
assuming Theorem \ref{the:intro-mira}. Since two cases are similar,
we only prove the first case. By the Hecke property of $\Aut_E$
\cf~\cite{FGV2} Proposition 1.5, the nontriviality of
$\RGC(\Bun_n^0,\mu_n^*\Aut_E)$ is equivalent to the nontriviality of
$\RGC(\Bun_n^d,\mu_n^*\Aut_E)$ for any $d\in\d{Z}$.

Recall that a vector bundle $\s{M}$ on $X$ is called {\em very
unstable} if it can be written as $\s{M}=\s{M}_1\oplus\s{M}_2$ with
$\s{M}_i\neq0$ and $\r{Ext}^1(\s{M}_1,\s{M}_2)=0$. By \cite{FGV2}
Lemma 3.3, for any line bundle $\s{L}$ on $X$, there exists an
integer $c_n(\s{L})$ such that if $d\geq c_n(\s{L})$ and
$\s{M}\in\Bun_n^d(k)$ with $\r{Hom}(\s{M},\s{L})\neq0$, then $\s{M}$
is very unstable. Define $\c{U}$ (resp. $\c{U}'$) to be the open
substack of $\Bun_n$ (resp. $\Bunp_n$) by the condition
$\r{Hom}(\s{M},\Omega_X^n)=0$ (resp.
$\r{Hom}(\s{M}',\Omega_{X'}^n)=0$). Then it is well-known that
$\c{U}\cap\Bun_n^d$ and $\c{U}'\cap\Bunp_n^d$ are of finite type for
any $d\in\d{Z}$. Now we fix an even integer $2d\geq
c_n(\Omega_{X'}^n)$ and let
$\W_E={\pi'}^*\Aut_E[d_1]|_{\Bunp^{2d}_n}$ by a cohomological shift
$d_1$ such that $\W_E\simeq{\pi'_n}_!\W_{E,n}^{2d}$ (see
\ref{num:steps-laumon}). Since $E$ is irreducible, $\Aut_E$ is
cuspidal and hence $\Aut_E|_{\Bunp_n^{2d}}$ is supported on
$\c{U}'\cap\Bunp_n^{2d}$ by \cite{FGV2} Lemma 9.4. Since
$\mu_n^{-1}(\c{U}')\subset\c{U}$, we have
 \begin{align*}
 \pi_!\ddot{\mu}_n^*\W_E\simeq\pi_!\ddot{\mu}_n^*{\pi'}^*\Aut_E[d_1]
 \simeq\pi_!\pi^*\mu_n^*\Aut_E[d_1]\simeq\pi_!\(\EL|_{\pi^{-1}(\c{U}\cap\Bun_n^d)}\)\otimes\mu_n^*\Aut_E[d_1].
 \end{align*}
By Serre duality, $\pi^{-1}(\c{U}\cap\Bun_n^d)$ is a vector bundle
of some rank $d_2>0$ over $\c{U}\cap\Bun_n^d$. Hence restricted on
$\c{U}\cap\Bun_n^d$, we have a distinguished triangle
 \begin{align*}
 \xymatrix{
 \pi_!\EL \ar[r] & \EL[-2d_2] \ar[r] & \EL \ar[r]^-{+1} &\:}
 \end{align*}
Tensoring with $\mu_n^*\Aut_E[d_1]$ and applying $\RGC$, we have a
distinguished triangle
 \begin{align*}
 \xymatrix{
 \RGC(\M_n^d,\ddot{\mu}_n^*\W_E) \ar[r] & \RGC(\Bun_n^d,\mu_n^*\Aut_E)[d_1-2d_2] \ar[r] &
 \RGC(\Bun_n^d,\mu_n^*\Aut_E)[d_1] \ar[r]^-{+1} &\:}
 \end{align*}
Hence $\RGC(\M_n^d,\ddot{\mu}_n^*\W_E)=0$ is equivalent to
$\RGC(\Bun_n^d,\mu_n^*\Aut_E)=0$ since
$\RGC(\Bun_n^d,\mu_n^*\Aut_E)$ is bounded from above.

Now if $\EL\subset\r{As}(E)$, then there exists $d\geq0$ such that
$\RG(X^{(d)},\r{As}(E)^{(d)})\neq0$ and $2d\geq c_n(\Omega_{X'}^n)$.
Hence by Theorem \ref{the:intro-mira} and the above argument,
$\RGC(\Bun_n^d,\mu_n^*\Aut_E)\neq0$ which implies
$\RGC(\Bun_n^0,\mu_n^*\Aut_E)\neq0$. Conversely, if
$\RGC(\Bun_n^0,\mu_n^*\Aut_E)\neq0$, then pick up some $d$ such that
$2d>\max\(c_n(\Omega_{X'}^n),n^2(4g-4)\)$. We have
$\RG(X^{(d)},\r{As}(E)^{(d)})\neq0$ for such $d$ which forces
$\EL\subset\r{As}(E)$ or
$\EL\subset\r{As}(E)^{\vee}\simeq\r{As}(E^{\vee})$. By Lemma
\ref{lem:asai-conjugate}, $E^{\vee}$ is conjugate orthogonal in the
later case. Hence $E$ is also conjugate orthogonal which, again by
the same lemma, implies that $\EL\subset\r{As}(E)$.\\
\end{num}

\begin{num}
{\em Proof of $(a)\Longleftrightarrow(c)$: Rankin-Selberg.} The
proof of the equivalence of $(a)$ and $(c)$ follows the same line as
above but replacing Theorem \ref{the:intro-mira} by the following
main result of Lysenko, which can be viewed as the split case (\ie,
$X'$ is disconnected).

\begin{theorem}[\cite{Ly}]
For any local systems $E_1$, $E_2$ of rank $n$ on $X$ and any
$d\geq0$, there is a canonical isomorphism:
 \begin{align*}
 \RGC(\M_n^d,\W_{E_1}\otimes\W_{E_2})\overset{\sim}{\lra}\RG(X^{(d)},(E_1\otimes
 E_2)^{(d)})[2d].
 \end{align*}\\
\end{theorem}

By the same argument in \ref{num:steps-1}, we have the following.

\begin{lem}
For irreducible $E_1$ and $E_2$,
$\RGC(\Bun_n^0,\Aut_{E_1}\otimes\Aut_{E_2})\neq0$ if and only if
$E_2\simeq E_1^{\vee}$.
\end{lem}

Applying the above lemma to $X=X'$, $E_1=E$ and $E_2=E^{\sigma}$, we
get the desired equivalence since it is easy to see that
$\sigma_n^*\Aut_E\simeq\Aut_{E^{\sigma}}$ by the
construction.\\
\end{num}

\begin{num}
{\em Proof of the rest.} The equivalence of $(a)$ and $(b)$ is due
to the Hecke property and the fact that
$\D\Aut_E\simeq\Aut_{E^{\vee}}$. Lemma \ref{lem:asai-conjugate}
implies that $(a)\Longleftrightarrow(d^+)\text{ or }(d^-)$. Hence
Main Theorem has been proved.

For Corollary \ref{cor:intro}, (1) is due to the fact that
$\r{As}(E)^{\vee}\simeq\r{As}(E^{\vee})$ and $c(E)=c(E^{\vee})$; (2)
is due to Lemma \ref{lem:asai-conjugate}.\\
\end{num}

\begin{num}\label{num:steps-laumon}
{\em Laumon's sheaf and Whittaker sheaf.} Let us briefly recall the
definition of Laumon's sheaf $\Lau^d_E$. Let $E$ be a local system
on $X'$ of rank $n$ and $d\geq0$. Denote $\Coh^d_n$ the stack
classifying coherent sheaves $\s{F}$ on $X$ of generic rank $n$ and
degree $d$. Inside $\Coh^d_0$, there is an open substack
$\Coh^d_{0,\leq m}$ by the additional condition that the restriction
of $\s{F}$ at any geometric point has dimension $\leq m$. Denote
$\Fl^d_0$ the stack classifying complete flags
$(0=\s{F}_0\subset\s{F}_1\subset\cdots\subset\s{F}_d)$ such that
$\s{F}_i/\s{F}_{i-1}$ is in $\Coh^1_0$. We have morphisms
$\f{p}:\Fl^d_0\ra\Coh^d_0$ remembering $\s{F}_d$ and
$\f{q}:\Fl^d_0\ra\(\Coh^1_0\)^d$ remembering
$(\s{F}_i/\s{F}_{i-1})_{i=1}^d$. Define $\Fl^d_{0,\leq
m}=\Fl^d_0\times_{\Coh^d_0}\Coh^d_{0,\leq m}$. Similarly, we define
for $X'$ the stacks $\Cohp^d_n$, $\Cohp_{0,\leq m}^d$, $\Flp^d_0$,
$\Flp^d_{0,\leq m}$ and $\f{p}'$, $\f{q}'$. Denote
$\r{div}^d_?:\Coh^d_{0,?}\ra X^{(d)}$ or
$\Cohp^d_{0,?}\ra{X'}^{(d)}$ the norm morphism and
$\r{div}_?=\r{div}^1_?$ for $?=\emptyset,\leq m$. Moreover, the
pullback under $\mu$ induces a morphism
$\mu_0:\Coh^d_0\ra\Cohp^{2d}_0$.

We have the following commutative diagram
 \begin{align*}
\xymatrix{
  \Cohp_0^d \ar[d]_-{\r{div}^d} & \Flp_0^d \ar[l]_-{\f{p}'}\ar[r]^-{\f{q}'} &
  \Cohp_0^1\times\cdots\times\Cohp_0^1 \ar[d]^-{\r{div}^{\times d}} \\
  {X'}^{(d)}  && {X'}^d \ar[ll]_-{\r{sym}}   }
 \end{align*}
and define {\em Springer's sheaf}
$\Spr_E^d:={\f{p}'}_!{\f{q}'}^*(\r{div}^{\times
d})^*E^{\boxtimes\:d}$ with a natural action by $\f{S}_d$ and {\em
Laumon's sheaf}
$\Lau_E^d:=\r{Hom}_{\f{S}_d}(\r{triv},\Spr_E^d)$ on $\Cohp_0^d$.\\

For $d\geq0$ and $n\geq1$, we introduce the stack $\Qb^d_n$
classifying the data $(\s{M},(s_i))$, where $\s{M}$ is a vector
bundle on $X$ of rank $n$ and degree $d$ and $s_i$ are injective
homomorphisms of coherent sheaves
 \begin{align*}
 s_i:\Omega_X^{(n-1)+\cdots+(n-i)}\lra\wedge^i\s{M},\qquad i=1,...,n
 \end{align*}
such that they satisfy the Pl\"{u}cker relations \cf~\cite{BG}
Section 1, \cite{FGV1} Section 2, \cite{FGV2} Section 4, or
\cite{Ly} Section 4 for details. Similarly, we define the stack
$\Qbp_n^d$ classifying the data $(\s{M}',(s'_i))$ but now on $X'$.
For our purpose, we need to introduce a twisted version of $\Qb^d_n$
as follows.

Since $\mu$ is \'{e}tale of degree $2$, we have a canonical
decomposition $\mu_*\s{O}_{X'}=\s{O}_X\oplus\s{L}_{\mu}$ for a line
bundle $\s{L}_{\mu}$ on $X$. We also have canonical isomorphisms
$\mu^*\Omega_X\simeq\Omega_{X'}$ and
$\mu^*\s{L}_{\mu}\simeq\s{O}_{X'}$. Let $\Qmb_n^d$ be the stack
classifying the data $(\s{M},(s_i))$ similar to the previous one but
now $s_i$ are injective homomorphisms of coherent sheaves
 \begin{align*}
 s_i:\Omega_X^{(n-1)+\cdots+(n-i)}\otimes\s{L}_{\mu}^{0+\cdots+(i-1)}\lra\wedge^i\s{M},\qquad
 i=1,...,n
 \end{align*}
still satisfying the Pl\"{u}cker relations. In fact, the stack
$\Qmb_n^d$ is nothing but the stack $\overline{\Bun}_N^{\c{F}_T}$
defined in \cite{FGV1} with the $T\simeq\r{GL}_1^n$-bundle $\c{F}_T$
on $X$ corresponding to the $m$-tuple of line bundles
 \begin{align*}
 \(\Omega_X^{n-1},\Omega_X^{n-2}\otimes\s{L}_{\mu},...,\Omega_X\otimes\s{L}_{\mu}^{n-2},\s{L}_{\mu}^{n-1}\).
 \end{align*}
The pullback under $\mu$ induces a closed embedding
$\overline{\mu}_n:\Qmb_n^d\ra\Qbp_n^{2d}$. Moreover, we have the
natural morphism $\pi_n:\Qmb_n^d\ra\M_n^d$ (resp.
$\pi'_n:\Qbp_n^d\ra\Mp_n^d$) by forgetting $s_2,...,s_n$ (resp.
$s'_2,...,s'_n$). Let $\f{p}_n=\pi\circ\pi_n:\Qmb_n^d\ra\Bun_n^d$
(resp. $\f{p}'_n=\pi'\circ\pi'_n:\Qbp_n^d\ra\Bunp_n^d$). We have
morphisms $\f{c}_n:\Qmb_n^d\ra\Coh_{0,\leq n}^d$ (resp.
$\f{c}'_n:\Qbp_n^d\ra\Cohp_{0,\leq n}^d$) sending $(\s{M},(s_i))$
(resp. $(\s{M}',(s'_i))$) to $\wedge^n\s{M}/\IM s_n$ (resp.
$\wedge^n\s{M}'/\IM s'_n$) and $\f{d}_n=\r{div}^d_{\leq
n}\circ\f{c}_n:\Qmb_n^d\ra X^{(d)}$ (resp. $\f{d}'_n=\r{div}_{\leq
n}^d\circ\f{c}'_n:\Qbp_n^d\ra{X'}^{(d)}$). Finally, let
$\Qb_n=\bigsqcup_{d\geq0}\Qb_n^d$,
$\Qbp_n=\bigsqcup_{d\geq0}\Qbp_n^d$ and
$\Qmb_n=\bigsqcup_{d\geq0}\:\Qmb_n^d$.

Inside $\Qmb_n^0$, there is an open substack $j:\Qm_n^0\hra\Qmb_n^0$
classifying the data $(\s{M}_i,r_i)$ where
$0=\s{M}_0\subset\s{M}_1\subset\cdots\subset\s{M}_n$ is a complete
flag of sub-vector bundles and
$r_i:\Omega_X^{n-i}\otimes\s{L}_{\mu}^{i-1}\overset{\sim}{\lra}\s{M}_i/\s{M}_{i-1}$
are isomorphisms for $i=1,...,n$. Similarly, one has $\Qp_n^0$ and
$j':\Qp_n^0\hra\Qbp_n^0$. We define the evaluation map
$\r{ev}:\Qp_n^0\ra\GA$ to be sum of the classes in
$\r{Ext}^1(\Omega_{X'}^{n-i},\Omega_{X'}^{n-i-1})\simeq\GA$ of the
extension
 \begin{align*}
 \xymatrix{
   0 \ar[r] & \s{M}'_i/\s{M}'_{i-1}
    \ar[r] & \s{M}'_{i+1}/\s{M}'_{i-1} \ar[r] &
   \s{M}'_{i+1}/\s{M}'_i \ar[d]^-{{s'_{i+1}}^{-1}}_{\wr} \ar[r] & 0   \\
   &\Omega_{X'}^{n-i}\ar[u]^-{s'_i}_{\wr}&&\Omega_{X'}^{n-i-1}&}
 \end{align*}
for $i=1,...,n-1$. Define $\ASb^0:={j'}_!\r{ev}^*\AS$ as a sheaf on
$\Qbp_n^0$.

Let $\Mod_n^d$ be the stack classifying modifications
$(\s{M}\hra\s{N})$ of rank-$n$ vector bundles on $X$ such that
$\deg(\s{N}/\s{M})=d$. We have morphisms
$\f{h}_{\leftarrow}:\Mod_n^d\ra\Bun_n$ (resp.
$\f{h}_{\ra}:\Mod_n^d\ra\Bun_n$) sending $(\s{M}\hra\s{N})$ to
$\s{M}$ (resp. $\s{N}$) and $\f{c}:\Mod_n^d\ra\Coh_0^d$ sending
$(\s{M}\hra\s{N})$ to $\s{N}/\s{M}$. Similarly, we define
$\Modp_n^d$, $\f{h}'_{\leftarrow}$, $\f{h}'_{\ra}$ and $\f{c}'$.
Consider the following commutative diagram \cf~\cite{FGV2} Section
4.3
 \begin{align*}
 \xymatrix{
   \Qbp_n^0 \ar[d]_-{\f{p}'_n}  & {\overline{\c{Z}}'}^d_n \ar[l]_-{\overline{\f{h}}'_{\leftarrow}}
   \ar[d]_-{\f{q}'_n} \ar[r]^-{\overline{\f{h}}'_{\ra}} & \Qbp_n^d \ar[d]^-{\f{p}'_n} \\
   \Bunp_n^0  & \Modp_n^d \ar[l]_-{\f{h}'_{\leftarrow}}\ar[r]^-{\f{h}'_{\ra}} & \Bunp_n^d   }
 \end{align*}
where ${\overline{\c{Z}}'}^d_n=\Qbp_n^0\times_{\Bunp_n^0}\Modp_n^d$.
The {\em Whittaker sheaf} $\W_{E,n}$ on $\Qbp_n$ is defined by the
formula
 \begin{align}\label{eqn:steps-whittaker}
 \W_{E,n}^d:=\W_{E,n}|_{\Qbp_n^d}:=\overline{\f{h}}'_{\ra!}\({\overline{\f{h}}'_{\leftarrow}}{}^*\(\ASb^0\)
 \otimes(\f{c}'\circ\f{q}'_n)^*\(\Lau_E^d\)\)[2q_n+dn]
 \end{align}
where $q_n:=\DIM\Qm_n^0$.\\

Notice that we have the following commutative diagram
 \begin{align*}
 \xymatrix{
   \Qbp^d_n \ar[d]_-{\pi'_n} \ar[r]^-{\f{d}'_n} & {X'}^{(d)} \ar[d]^-{\alpha} \\
   \Mp^d_n \ar[r]^-{\f{d}'} & \r{Pic}^d_{X'}   }
 \end{align*}
where $\alpha$ is the Abel-Jacobi map and
$\f{d}:\M^d_n\ra\r{Pic}^d_X$ (resp. $\f{d}':\Mp^d_n\ra\r{Pic}_X^d$)
sends $(\s{M},s_1)$ (resp. $(\s{M}',s'_1)$) to the line bundle
$\s{M}\otimes\Omega_X^{-\frac{n(n-1)}{2}}$ (resp.
$\s{M}'\otimes\Omega_{X'}^{-\frac{n(n-1)}{2}}$). Hence there is a
morphism
 \begin{align*}
 \delta'_n:=\f{d}'_n\times\f{d}:\Qbp^{2d}_n\times_{\Mp^{2d}_n}\M^d_n\lra
 {X'}^{(2d)}\times_{\r{Pic}^{2d}_{X'}}\r{Pic}^d_X
 \end{align*}
and a closed embedding $\iota:=(\mu^{(d)},\alpha):X^{(d)}\ra
{X'}^{(2d)}\times_{\r{Pic}^{2d}_{X'}}\r{Pic}^d_X$. The rest sections
are devoted to prove the following.
\begin{proposition}\label{pro:steps-whittaker}
For any local systems $E$ on $X'$ of rank $n$ and $L$ on $X$ of rank
$1$, we have a canonical isomorphism
 \begin{align*}
 {\delta'_n}_!\(\W_{E,n}^{2d}\boxtimes\f{d}^*\sf{A}_L\)\overset{\sim}{\lra}
 \iota_*\(\r{As}(E)\otimes L\)^{(d)}[2d]
 \end{align*}
on ${X'}^{(2d)}\times_{\r{Pic}^{2d}_{X'}}\r{Pic}^d_X$ for all
$d\geq0$.\\
\end{proposition}
\end{num}

\begin{num}
{\em Proof of Theorem \ref{the:intro-mira}.} Assuming Proposition
\ref{pro:steps-whittaker}, the proof is immediate once realizing the
following commutative diagram
 \begin{align*}
 \xymatrix{
   \Qbp^{2d}_n\times_{\Mp^{2d}_n}\M^d_n \ar[d]_-{\delta'_n=\f{d}'_n\times\f{d}}
   \ar[rr]^-{\pi'_n\times\r{id}} && \Mp^{2d}_n\times_{\Mp^{2d}_n}\M^d_n=\M^d_n \ar[d]^-{\f{d}} \\
   {X'}^{(2d)}\times_{\r{Pic}^{2d}_{X'}}\r{Pic}^d_X \ar[rr]^-{\alpha\times\r{id}}
   && \r{Pic}^{2d}_{X'}\times_{\r{Pic}^{2d}_{X'}}\r{Pic}^d_X=\r{Pic}^d_X  }
 \end{align*}
and that $\W_E\simeq{\pi'_n}_!\W_{E,n}^{2d}$.\\
\end{num}

\section{Restriction of the Whittaker sheaf}
\label{sec:res}

Let $m\geq1$ be an arbitrary integer and define the Whittaker sheaf
$\W_{E,m}$ by the same formula \eqref{eqn:steps-whittaker} but
replacing $n$ by $m$ (where $E$ is still of rank $n$). Consider the
following commutative diagram
 \begin{align*}
 \xymatrix{
   \Qmb^d_m \ar[d]_-{\f{d}_m} \ar@/^1pc/[rrd]_-{\delta_m} \ar[rr]^-{(\overline{\mu}_m,\pi_m)}
   && \Qbp^{2d}_m\times_{\Mp^{2d}_m}\M^d_m \ar[d]^-{\delta'_m} \\
   X^{(d)} \ar[rr]^-{\iota} &&   {X'}^{(2d)}\times_{\r{Pic}^{2d}_{X'}}\r{Pic}^d_X }
 \end{align*}
and let $\nu_m=(\overline{\mu}_m,\pi_m)$,
$\delta_m=\delta'_m\circ\nu_m$. Proposition
\ref{pro:steps-whittaker} will follow from the following two
propositions and this section is dedicated to prove the first one.

\begin{proposition}\label{pro:res-res}
For any local systems $E$ on $X'$ of rank $n$ and $L$ on $X$ of rank
$1$, $m\geq1$, $d\geq0$, the natural map
 \begin{align*}
 {\delta'_m}_!\(\W_{E,m}^{2d}\boxtimes\f{d}^*\sf{A}_L\)\lra\delta_{m!}\nu_m^*\(\W_{E,m}^{2d}\boxtimes\f{d}^*\sf{A}_L\)
 \end{align*}
is an isomorphism.\\
\end{proposition}

\begin{proposition}\label{pro:res-coh}
Let notations be as above, there is a canonical isomorphism
 \begin{align*}
 \f{d}_{m!}\nu_m^*\(\W_{E,m}^{2d}\boxtimes\f{d}^*\sf{A}_L\)\overset{\sim}{\lra}(\r{As}(E)\otimes
 L)^{(d)}_m[2d]
 \end{align*}
on $X^{(d)}$, where $0=(\r{As}(E)\otimes
 L)^{(d)}_0\subset(\r{As}(E)\otimes
 L)^{(d)}_1\subset\cdots\subset(\r{As}(E)\otimes
 L)^{(d)}_m\subset\cdots$ is a filtration of $(\r{As}(E)\otimes
 L)^{(d)}$ such that $(\r{As}(E)\otimes
 L)^{(d)}_n=(\r{As}(E)\otimes
 L)^{(d)}$.\\
\end{proposition}

\begin{num}
{\em Stratifications-I.} To proceed, we recall stratifications
defined on $\Qbp_m^d$ and $\Qmb_m^d$, respectively.

For any $\lambda\in\Lambda_{m,\r{pos}}^d$, let $\Qbp_m^{\lambda}$ be
the stack classifying the data $(\s{M}',(s'_i),(D'_i))$ where
$\s{M}'$ is a vector bundle on $X'$ of rank $m$,
$(D'_1,D'_1+D'_2,...,D'_1+\cdots+D'_m)\in{X'}^{\lambda}$
\cf~\ref{num:intro-notations}, and $s'_i$ ($i=1,...,m$) is an
inclusion of {\em vector bundles}
 \begin{align*}
 s'_i:\Omega_{X'}^{(n-1)+\cdots+(n-i)}(D'_1+\cdots+D'_i)\lra\wedge^i\s{M}'
 \end{align*}
such that $(s'_1,...,s'_m)$ satisfies the Pl\"{u}cker relations. We
have a natural morphism $\Qbp_m^{\lambda}\ra\Qbp_m^d$. It was shown
in \cite{BG} that each $\Qbp_m^{\lambda}$ becomes a locally closed
substack and they together form a stratification of $\Qbp_m^d$ for
all $\lambda\in\Lambda_{m,\r{pos}}^d$. We have a natural morphism
$\f{s}':\Qbp_m^{\lambda}\ra{X'}^{\lambda}$ by remembering
$(D'_1,D'_1+D'_2,...,D'_1+\cdots+D'_m)$ and define
$\Qbp_{m,?}^{\lambda}=\Qbp_m^{\lambda}\times_{{X'}^{\lambda}}{X'}^{\lambda}_?$
which are closed substacks of $\Qbp_m$,
$\f{s}'_?=\f{s}'|_{\Qbp_{m,?}^{\lambda}}$ for $?=\r{eff},+,-$.

For any $\lambda\in\Lambda_{m,\r{pos}}^d$, let $\Qmb_m^{\lambda}$ be
the stack classifying the data $(\s{M},(s_i),(D_i))$ where $\s{M}$
is a vector bundle on $X$ of rank $m$, $(D_1,...,D_1+\cdots+D_m)\in
X^{\lambda}$, and $s_i$ is an inclusion of vector bundles
 \begin{align*}
 s_i:\Omega_X^{(n-1)+\cdots+(n-i)}\otimes\s{L}_{\mu}^{0+\cdots+(i-1)}(D_1+\cdots+D_i)\lra\wedge^i\s{M}
 \end{align*}
such that $(s_1,...,s_m)$ satisfies the Pl\"{u}cker relations. As in
the previous case, we have a natural morphism
$\Qmb_m^{\lambda}\ra\Qmb_m^d$ such that $\Qmb_m^{\lambda}$ becomes a
locally closed substack, and they together form a stratification of
$\Qmb_m^d$. We have a natural morphism $\f{s}:\Qmb_m^{\lambda}\ra
X^{\lambda}$ and define
$\Qmb_{m,?}^{\lambda}=\Qmb_m^{\lambda}\times_{X^{\lambda}}X^{\lambda}_?$,
which are closed substacks of $\Qmb_m$,
$\f{s}_?=\f{s}|_{\Qmb_{m,?}^{\lambda}}$ for $?=\r{eff},+,-$.

The pullback under $\mu$ again induces a natural morphism
$\overline{\mu}_m^{\lambda}:\Qmb_{m,?}^{\lambda}\ra\Qbp_{m,?}^{2\lambda}$
and hence
$\nu_m^{\lambda}:\Qmb_{m,?}^{\lambda}\ra\Qbp_{m,?}^{2\lambda}\times_{\Mp^{2d}_m}\M^d_m$
for $?=\emptyset,\r{eff},+,-$ which are all closed embeddings, such
that the following diagram commutes
 \begin{align*}
 \xymatrix{
   \Qmb_{m,?} \ar@/_4pc/[dd]^-{\f{d}_m} \ar[d]_-{\f{s}} \ar[rr]^-{\nu_m^{\lambda}} &&
   \Qbp_{m,?}^{2\lambda}\times_{\Mp^{2d}_m}\M^d_m \ar[d]_-{\f{s}'\times\f{d}} \ar@/^4pc/[dd]^-{\f{d}'_m\times\f{d}'}\\
   X^{\lambda}_? \ar[d]_-{\r{sym}^{\lambda}} \ar[rr]^-{(\mu^{\lambda},\alpha\circ\r{sym}^{\lambda})}
   && {X'}_?^{2\lambda}\times_{\r{Pic}_{X'}^{2d}}\r{Pic}_X^d \ar[d]_-{\r{sym}^{2\lambda}\times\r{id}}\\
   X^{(d)} \ar[rr]^-{\iota} && {X'}^{(2d)}\times_{\r{Pic}_{X'}^{2d}}\r{Pic}_X^d. }
 \end{align*}

For $\lambda\in\Lambda_{m,-}^d$, the stack $\Qbp_{m,-}^{\lambda}$
will be nonempty and it equivalently classifies the data
$((\s{M}'_i),(r'_i),(D'_i))$ where
$0=\s{M}'_0\subset\s{M}'_1\subset\cdots\subset\s{M}'_m=\s{M}'$ is a
complete flag of sub-vector bundles of $\s{M}'$, $D'_i$ is as above
but with $0\leq D'_1\leq\cdots\leq D'_m$, and $r'_i$ is an
isomorphism
 \begin{align*}
 r'_i:\Omega_{X'}^{n-i}(D'_i)\overset{\sim}{\lra}\s{M}'_i/\s{M}'_{i-1}
 \end{align*}
for $i=1,...,m$. Let $\r{ev}^{\lambda}:\Qbp_{m,-}^{\lambda}\ra\GA$
be the morphism sending the above data to the sum of $m-1$ classes
in
$\r{Ext}^1(\Omega_{X'}^{n-i-1}(D'_i),\Omega_{X'}^{n-i}(D'_i))\simeq\GA$
corresponding to the pullbacks of the extensions
 \begin{align*}
 \xymatrix{
   0 \ar[r] & \s{M}'_i/\s{M}'_{i-1}
    \ar[r] & \s{M}'_{i+1}/\s{M}'_{i-1} \ar[r] &
   \s{M}'_{i+1}/\s{M}'_i \ar[d]^-{{s''_{i+1}}^{-1}}_{\wr} \ar[r] & 0   \\
   &\Omega_{X'}^{n-i}(D'_i)\ar[u]^-{s''_i}_{\wr}&&\Omega_{X'}^{n-i-1}(D'_{i+1})&}
 \end{align*}
under the inclusion
$\Omega_{X'}^{n-i-1}(D'_i)\hra\Omega_{X'}^{n-i-1}(D'_{i+1})$ for
$i=1,...,m-1$. Finally, let $\ASb^{\lambda}=\r{ev}^{\lambda*}\AS$ be
a local system on $\Qbp_{m,-}^{\lambda}$.\\
\end{num}

\begin{num}
{\em Pullback of Laumon's sheaf.} For $\lambda\in\Lambda_{m,-}^d$,
the scheme $X^{\lambda}_-$ classifies $(D_1,...,D_m)$ where $D_i$ is
a divisor on $X$ of degree $\lambda_i$ and $0\leq D_1\leq\cdots\leq
D_m$. Let $\f{c}^{\lambda}:X^{\lambda}_-\ra\Coh^d_{0,\leq m}$ be the
morphism sending $(D_1,...,D_m)$ to
 \begin{align*}
 \Omega_X^{m-1}(D_1)/\Omega_X^{m-1}\oplus\Omega_X^{m-2}\otimes\s{L}_{\mu}(D_2)/\Omega_X^{m-2}\otimes\s{L}_{\mu}
 \oplus\cdots\oplus\s{L}_{\mu}^{m-1}(D_m)/\s{L}_{\mu}^{m-1}.
 \end{align*}
Similarly, we have a morphism
${\f{c}'}^{\lambda}:{X'}^{\lambda}_-\ra\Cohp^d_{0,\leq m}$ sending
$(D'_1,...,D'_m)$ to
 \begin{align*}
 \Omega_{X'}^{m-1}(D'_1)/\Omega_{X'}^{m-1}\oplus\Omega_{X'}^{m-2}(D'_2)/\Omega_{X'}^{m-2}
 \oplus\cdots\oplus\s{O}_{X'}^{m-1}(D'_m)/\s{O}_{X'}^{m-1}.
 \end{align*}
By \cite{La1} Th\'{e}or\`{e}me 3.3.8, the restriction
${\f{c}'}^{\lambda*}\Lau_E^d$ of Laumon's sheaf has maximal
(possible) cohomological degree (with respect to the usual
t-structure) $2d(\lambda)$, where
$d(\lambda):=\sum_{i=1}^m(m-i)\lambda_i$. Its highest cohomology
sheaf
 \begin{align*}
 E^{\lambda}_-:=\r{H}^{2d(\lambda)}{\f{c}'}^{\lambda*}\Lau_E^d
 \end{align*}
does not vanish if and only if $\lambda_1=\cdots=\lambda_{m-n}=0$
(which is an empty condition if $m\leq n$). See \cite{La1} for a
description of the stalks of $E^{\lambda}_-$. The following
proposition proved by Lysenko \cite{Ly} Proposition 2 will play a
key role later, of which the proof uses the geometric
Casselman-Shalike formula for general linear groups \cf~\cite{FGV1},
\cite{Ng} and \cite{NP}.

\begin{proposition}\label{pro:res-lysenko}\footnote[2]{A similar result on the restriction of Whittaker sheaves on
different strata is proved in \cite{FGV2} Proposition 4.12.} The
restriction $\W_{E,m}^d|_{\Qbp_m^{\lambda}}$ is supported on
$\Qbp_{m,-}^{\lambda}$, and its further restriction to the later one
is isomorphic to
 \begin{align*}
 \ASb^{\lambda}\otimes{\f{s}'_-}{}^*E^{\lambda}_-[2q_m+md-2d(\lambda)].\\
 \end{align*}
\end{proposition}

The following simple lemma will be important in the later proof.

\begin{lem}\label{lem:res-key}
Let $\s{F}$ be a coherent sheaf on $X$ and consider the paring
 \begin{align*}
 \r{ev}_{\s{F}}:\:&\r{Ext}^1(\s{F},\Omega_X)\times\r{H}^0(X',\mu^*\s{F})\overset{i^1_+\times\r{id}}{\lra}
 \r{Ext}^1(\mu^*\s{F},\Omega_{X'})\times\r{H}^0(X',\mu^*\s{F})\\
 &\lra\r{H}^1(X',\Omega_{X'})\simeq\GA.
 \end{align*}
Then
 \begin{align*}
 \r{pr}_{2!}{\r{ev}_{\s{F}}}^*\AS={i^0_-}_!\EL[-2\DIM\r{H}^0(X,\s{F})]
 \end{align*}
where $\r{pr}_i$ is the projection to the $i$-th factor,
$i^1_+:\r{Ext}^1(\s{F},\Omega_X)\hra\r{Ext}^1(\mu^*\s{F},\Omega_{X'})$
and $i^0_-:\r{Hom}(\s{L}_{\mu},\s{F})\hra\r{H}^0(X',\mu^*\s{F})$ are
closed embeddings induced by $\mu$.
\end{lem}
\begin{proof}
We only need to show that the orthogonal complement of
$i^1_+\r{Ext}^1(\s{F},\Omega_X)\subset\r{Ext}^1(\mu^*\s{F},\Omega_{X'})$
inside $\r{H}^0(X',\mu^*\s{F})$ under the canonical pairing is
$i^0_-\r{Hom}(\s{L}_{\mu},\s{F})$. The paring between these two
subspaces is zero due to the fact $\r{H}^0(X,\s{L}_{\mu})=0$, and
they are orthogonal complement of each other because they have the
complimentary dimensions since
$\r{H}^0(X',\mu^*\s{F})=\r{H}^0(X,\s{F})\oplus\r{Hom}(\s{L}_{\mu},\s{F})$.\\
\end{proof}
\end{num}

\begin{num}\label{num:res-res}
{\em Proof of Proposition \ref{pro:res-res}.} By Proposition
\ref{pro:res-lysenko}, it amounts to prove the following assertion.
For $\rho\in\Lambda_{m,-}^{2d}$, the direct image
$(\f{s}'\times\f{d}')_!\ASb^{\rho}$ vanishes unless $\rho=2\lambda$
and then the natural map
 \begin{align*}
 (\f{s}'\times\f{d}')_!\ASb^{2\lambda}\lra
 (\f{s}'\times\f{d}')_!{\nu_{m}^{\lambda}}_!{\nu_m^{\lambda}}{}^*\ASb^{2\lambda}
 \end{align*}
is an isomorphism.\\

For $j=1,...,m$, we introduce the stack $\c{P}_{m,j}^{\rho}$
classifying the data
 \begin{align*}
 (\s{M},(\s{M}_i)_{i=1}^j,(\s{M}'_i)_{i=j+1}^{m-1};
 (D_i)_{i=1}^j,(D'_i)_{i=j+1}^m;(r_i)_{i=1}^j,(r'_i)_{i=j+1}^m)
 \end{align*}
where
$0\subset\mu^*\s{M}_1\subset\cdots\subset\mu^*\s{M}_j\subset\s{M}'_{j+1}
\subset\cdots\subset\s{M}'_{m-1}\subset\mu^*\s{M}$ is a complete
flag of sub-vector bundles of $\mu^*\s{M}$ where $\s{M}$ is a vector
bundle on $X$ of rank $m$; $0\leq\mu^*D_1\leq\cdots\leq\mu^*D_j\leq
D'_{j+1}\leq\cdots\leq D'_m$ is in ${X'}^{\rho}_-$
($\mu^*D:=\mu^{-1}(D)$) and
$r_i:\Omega_X^{m-i}\otimes\s{L}_{\mu}^{i-1}(D_i)\overset{\sim}{\lra}\s{M}_i/\s{M}_{i-1}$
for $i=1,...,j$,
$r'_i:\Omega_{X'}^{m-i}(D'_i)\overset{\sim}{\lra}\s{M}'_i/\s{M}'_{i-1}$
for $i=j+1,...,m$ ($\s{M}_0=0,\s{M}'_m=\mu^*\s{M}$). Hence
$\c{P}_{m,j}^{\rho}$ is empty if $2\nmid\rho_i$ for some $1\leq
i\leq j$. We have the following successive closed embeddings
 \begin{align*}
 \c{P}^{\rho}_{m,m}\hra\c{P}_{m,m-1}^{\rho}\hra\cdots\hra\c{P}^{\rho}_{m,1}=\Qbp^{\rho}_{m,-}\times_{\Mp_m^{2d}}\M_m^d
 \end{align*}
and $\c{P}^{\rho}_{m,m}$ is empty unless $\rho=2\lambda$ in which
case $\c{P}^{2\lambda}_{m,m}=\Qmb^{\lambda}_m$. Write
$\nu_{m,j}^{\rho}$ for the inclusion
$\c{P}_{m,j}^{\rho}\hra\Qbp^{\rho}_{m,-}\times_{\Mp_m^{2d}}\M_m^d$.
We prove successively that the natural map
 \begin{align*}
 (\f{s}'\times\f{d}')_!{\nu^{\rho}_{m,j}}_!{\nu^{\rho}_{m,j}}{}^*\ASb^{\rho}\lra
 (\f{s}'\times\f{d}')_!{\nu^{\rho}_{m,j+1}}_!{\nu^{\rho}_{m,j+1}}{}^*\ASb^{\rho}
 \end{align*}
is an isomorphism for $j=1,...,m-1$.

In fact, let $\overline{\c{R}}_{m,j}^{\rho}$ be the stack
classifying
$(\s{N},(\s{N}'_i)_{i=j+1}^{m-1};(D_i)_{i=1}^j,(D'_i)_{i=j+1}^m;(\widetilde{r}'_i)_{i=j+1}^m)$
where
$0\subset\s{N}'_{j+1}\subset\cdots\subset\s{N}'_{m-1}\subset\mu^*\s{N}$
is a complete flag of sub-vector bundles of $\mu^*\s{N}$ where
$\s{N}$ is a vector bundle on $X$ of rank $m-j$,
$(D_i)_{i=1}^j,(D'_i)_{i=j+1}^m$ are as above and
$\widetilde{r}'_i:\Omega_{X'}^{m-i}\overset{\sim}{\lra}\s{N}'_{i+1}/\s{N}'_i$
($\s{N}'_j=0,\s{N}'_m=\mu^*\s{N}$). The closed substack
$\c{R}_{m,j}^{\rho}\hra\overline{\c{R}}_{m,j}^{\rho}$ is defined by
the condition that
$\widetilde{r}'_{j+1}:\Omega_{X'}^{m-j-1}(D'_{j+1})\overset{\sim}{\lra}\s{N}'_{j+1}$
coincides with the pullback of
$\widetilde{r}_{j+1}:\Omega_X^{m-j-1}\otimes\s{L}_{\mu}^j(D_{j+1})\overset{\sim}{\lra}\s{N}_{j+1}$
for some divisor $D_{j+1}$ on $X$ with $\mu^*D_{j+1}=D'_{j+1}$.
There are natural morphisms
$\overline{\f{f}}:\c{P}^{\rho}_{m,j}\ra\overline{\c{R}}^{\rho}_{m,j}$
and $\f{f}:\c{P}^{\rho}_{m,j+1}\ra\c{R}^{\rho}_{m,j}$ via taking the
quotient by $\s{M}_j$, which are generalized affine fibration
(\cf~\cite{Ly} 0.1.1 for the convention). We have the following
commutative diagram
 \begin{align*}
 \xymatrix{
   \c{P}^{\rho}_{m,j+1} \ar[d]_-{\f{f}} \ar@{^(->}[r] &  \c{P}^{\rho}_{m,j}
   \ar[d]^-{\overline{\f{f}}}\ar[dr]^-{\f{s}'\times\f{d}'} \\
   \c{R}^{\rho}_{m,j} \ar@{^(->}[r] &  \overline{\c{R}}^{\rho}_{m,j}
   \ar[r] & {X'}^{\rho}_-\times_{\r{Pic}_{X'}^{2d}}\r{Pic}_X^d  }
 \end{align*}
in which the square is Cartesian. Applying Lemma \ref{lem:res-key}
to the vector bundle
$\s{N}\otimes\Omega_X^{-m+j+1}\otimes\s{L}_{\mu}^{-j+1}$, one easily
see that $\overline{\f{f}}_!{\nu^{\rho}_{m,j}}{}^*\ASb^{\rho}$ is
supported on $\c{R}^{\rho}_{m,j}$, and our assertion follows. Hence
Proposition \ref{pro:res-res} is proved. Moreover, it has the
following corollary.\\
\end{num}

\begin{corollary}\label{cor:res-grade}
Let $E$ and $L$ be as in Proposition \ref{pro:res-res}, the complex
$\f{d}_{m!}\nu_m^*(\W_{E,m}^{2d}\boxtimes\f{d}^*\sf{A}_L)[-2d]$ on
$X^{(d)}$ is placed in degree zero. It has a canonical filtration by
constructible subsheaves such that the direct sum of all graded
terms
 \begin{align*}
 \r{gr}\:\f{d}_{m!}\nu_m^*(\W_{E,m}^{2d}\boxtimes\f{d}^*\sf{A}_L)[-2d]\simeq\bigoplus_{\lambda\in\Lambda_{m,-}^d}
 \(\r{sym}^{\lambda}_{\;!}{\mu^{\lambda}}^*E^{2\lambda}_-\)\otimes
 L^{(d)}
 \end{align*}
In particular,
$\r{gr}\:\f{d}_{m!}\nu_m^*(\W_{E,m}^{2d}\boxtimes\f{d}^*\sf{A}_L)[-2d]\simeq
\r{gr}\:\f{d}_{n!}\nu_n^*(\W_{E,n}^{2d}\boxtimes\f{d}^*\sf{A}_L)[-2d]$
for $m\geq n$.
\end{corollary}
\begin{proof}
Since the morphism $\f{s}:\Qmb_{m,-}^{\lambda}\ra X^{\lambda}_-$ is
a generalized affine fibration of rank $q_m+md-d-2d(\lambda)$, the
assertion follows from Proposition \ref{pro:res-lysenko},
\ref{num:res-res}, Lemma \ref{lem:res-key}, and the Cousin spectral
sequence for the
computation of the $!$-direct image via stratifications.\\
\end{proof}

\section{Reduction to the moduli of torsion flags}
\label{sec:red}

\begin{num}
{\em More stacks.} Recall that we have stacks $\Qp_m^0\hra\Qbp^0_m$
and
$\overline{\f{h}}'_{\ra}:{\overline{\c{Z}}'}^{2d}_m=\Qbp^0_m\times_{\Bunp_m}\Modp^{2d}_m\ra\Qbp^{2d}_m$.
Define
${\c{Z}'}^{2d}_m=\Qp^0_m\times_{\Bunp_m}\Modp^{2d}_m\hra{\overline{\c{Z}}'}^{2d}_m$.
Similarly, we define $\Zm^d_m=\Qm^0_m\times_{\Bun_m}\Mod_m^d$.

Define $\c{C}^{2d}_m$ to be the stack classifying the data
$(\s{N},(\s{M}'_i),(r'_i))$ where $\s{N}$ is in $\Coh_m^d$,
$0=\s{M}'_0\subset\s{M}'_1\subset\cdots\subset\s{M}'_m\subset\mu^*\s{N}$
and
$r'_i:\Omega_{X'}^{m-i}\overset{\sim}{\lra}\s{M}'_i/\s{M}'_{i-1}$
for $i=1,...,m$ are isomorphisms. Define $\Cc^{2d}_m$ to be the open
substack of $\c{C}^{2d}_m$ by the condition that $\s{N}$ is locally
free. Then $\Cc^{2d}_m$ is naturally identified with
${\c{Z}'}_m^{2d}\times_{\Bunp_m}\Bun_m$. Now consider the subfunctor
$\c{Y}^{2d}_m$ of $\c{C}^{2d}_m$ defined by the following way. For
any scheme $S$,
$(\s{N}_S,(\s{M}'_{i,S}),(r'_{i,S}))\in\r{Hom}(S,\c{C}^{2d}_m)$ is
in $\r{Hom}(S,\c{Y}^{2d}_m)$ if there exists
$(\s{M}_{i,S},(r_{i,S}))$, where $\s{M}_{i,S}$ is a vector bundle on
$S\times X$ of rank $i$,
$0=\s{M}_{0,S}\subset\s{M}_{1,S}\subset\cdots\subset\s{M}_{m,S}\subset\s{N}_S$
and
$r_{i,S}:\s{O}_S\boxtimes(\Omega_X^{m-i}\otimes\s{L}_{\mu}^{i-1})\overset{\sim}{\lra}\s{M}_{i,S}/\s{M}_{i-1,S}$
for $i=1,...,m$ are isomorphisms, such that
$r'_{i,S}:\s{O}_S\boxtimes\Omega_{X'}^{m-i}\overset{\sim}{\lra}\s{M}'_{i,S}/\s{M}'_{i-1,S}$
coincides with $\mu^*r_{i,S}$ on $S\times
X'-(\r{id}\times\mu)^{-1}T$ for a closed subscheme $T$ of $S\times
X$ which is finite over $S$. We have
 \begin{lem}
 The subfunctor $\c{Y}^{2d}_m\hra\c{C}^{2d}_m$ is a closed
 embedding, hence $\c{Y}^{2d}_m$ is an Artin stack.
 \end{lem}
 \begin{proof}
 An object
 $(\s{N}_S,(\s{M}'_{i,S}),(r'_{i,S}))\in\r{Hom}(S,\c{C}^{2d}_m)$
 will induce maps
  \begin{align*}
  s'_{i,S}:\s{O}_S\boxtimes\Omega_{X'}^{(m-1)+\cdots+(m-i)}\lra\wedge^i\mu^*\s{N}=\mu^*\wedge^i\s{N}
  \end{align*}
 for $i=1,...,m$. Let
 \begin{align*}
 {s'_{i,S}}{}^{\sigma}:\s{O}_S\boxtimes\Omega_{X'}^{(m-1)+\cdots+(m-i)}
 \simeq\s{O}_S\boxtimes\sigma^*\Omega_{X'}^{(m-1)+\cdots+(m-i)}
 \overset{\sigma^*s'_{i,S}}{\lra}\sigma^*\mu^*\wedge^i\s{N}\simeq\mu^*\wedge^i\s{N}
 \end{align*}
 Then
 $(\s{N}_S,(\s{M}'_{i,S}),(r'_{i,S}))\in\r{Hom}(S,\c{Y}^{2d}_m)$ if
 and only if the support of $s'_{i,S}-(-1)^{\frac{i(i-1)}{2}}{s'_{i,S}}{}^{\sigma}$ is a closed
 subscheme of $S\times X'$ finite over $S$ for $i=1,...,m$. Since
 $\s{N}$ is flat over $S$, it is locally free outside a closed
 subscheme of $S\times X$ finite over $S$. Hence the assertion
 follows from \cite{Ly} Sublemma 4.\\
 \end{proof}

For $(\s{N}_S,(\s{M}'_{i,S}),(r'_{i,S}))\in\r{Hom}(S,\c{Y}^{2d}_m)$,
the induced map
 \begin{align*}
 s^{\sharp}_S:\s{O}_S\boxtimes\Omega_{X'}^{\frac{m(m-1)}{2}}\overset{\sim}{\lra}\det\s{M}'_m\hra\det\mu^*\s{N}
 \end{align*}
satisfies that
$s^{\sharp}_S=(-1)^{\frac{m(m-1)}{2}}(s^{\sharp}_S)^{\sigma}$. Hence
the zero divisor of $s^{\sharp}-S$ locates in the closed subscheme
$S\times X^{(d)}$ of $S\times X'^{(2d)}$. It induces a morphism
$\b{d}_m:\c{Y}^{2d}_m\ra X^{(d)}$. Define
$\Yc^{2d}_m=\Cc^{2d}_m\cap\Y^{2d}_m\hra\Y^{2d}_m$ which is an open
substack naturally identified with
${\c{Z}'}^{2d}_m\times_{\Qbp^{2d}_m}\Qmb^d_m$. Moreover, the
restriction of $\b{d}_m$ to
$\Yc^{2d}_m={\c{Z}'}^{2d}_m\times_{\Qbp^{2d}_m}\Qmb^d_m$ coincides
with $\r{id}\times\f{d}_m$.

For $j=0,...,m$, we define a subfunctor $\c{Y}_{m,j}^{2d}$ of
$\c{Y}^{2d}_m$ by the condition that the partial data
$((\s{M}'_i)_{i=1}^j,(r'_i)_{i=1}^j)$ is a pullback of
$((\s{M}_i)_{i=1}^j,(r_i)_{i=1}^j)$ (on the whole $S\times X$) in
the above sense. Then we have a successive closed embeddings
 \begin{align*}
 \c{Y}_{m,m}^{2d}\hra\c{Y}_{m,m-1}^{2d}\hra\cdots\hra\c{Y}_{m,0}^{2d}=\c{Y}_m^{2d}
 \end{align*}
and natural morphisms $\c{Y}_{m,j}^{2d}\ra\c{Y}_{m-j}^{2d}$ via
taking the quotient by $\s{M}'_j$, which are generalized affine
fibration. Let
$\Yc^{2d}_{m,j}=\Cc^{2d}_m\cap\Y^{2d}_{m,j}\hra\Y^{2d}_{m,j}$ be the
open substack, then $\Yc^{2d}_{m,m}$ is naturally identified with
$\Zm^d_m$.\\
\end{num}

\begin{num}
{\em Iterated modifications.} We have a natural morphism
$\c{Y}_m^{2d}\hra\c{C}_m^{2d}\ra\Cohp^{2d}_m$ sending
$(\s{N},(\s{M}'_i),(r'_i))$ to $\mu^*\s{N}/\s{M}'_m$, and define
$\Yt_m^{2d}=\c{Y}_m^{2d}\times_{\Cohp_0^{2d}}\Flp^{2d}_0$,
$\Ytc^{2d}_m=\Yc_m^{2d}\times_{\Cohp_0^{2d}}\Flp^{2d}_0$. Denote
$\b{c}_m:\c{Y}_m^{2d}\ra\Cohp^{2d}_m\times_{{X'}^{(2d)}}X^{(d)}$ and
$\widetilde{\b{c}}_m:\Yt_m^{2d}\ra\Flp^{2d}_0\times_{{X'}^{(2d)}}X^{(d)}$
the induced morphisms from $\b{d}_m$. We have the natural projection
$\b{p}_m:\Yt_m^{2d}\ra\c{Y}_m^{2d}$. Moreover, for $j=0,...,m$, we
define
$\Yt^{2d}_{m,j}=\c{Y}^{2d}_{m,j}\times_{\c{Y}^{2d}_m}\Yt^{2d}_m$ and
$\b{p}_{m,j}:\Yt^{2d}_{m,j}\lra\c{Y}^{2d}_{m,j}$ the natural
projection. We have the following commutative diagram
 \begin{align*}
 \xymatrix{
   \Yt^{2d}_{m,j} \ar[d]_-{\b{p}_{m,j}} \ar@{^(->}[r]_-{\widetilde{\b{u}}_{m,j}}
   \ar@/^3pc/[rrrr]^-{\widetilde{\b{d}}_{m,j}}
   & \Yt^{2d}_m \ar[d]_-{\b{p}_m} \ar[r]_-{\widetilde{\b{c}}_m}
   \ar@/^1.5pc/[rrr]^-{\widetilde{\b{d}}_m}
   & \Flp^{2d}_0\times_{{X'}^{(2d)}}X^{(d)} \ar[d]_-{\f{p}'\times\r{id}}
   \ar[rr]_-{(\r{div}^{\times 2d}\circ\f{q}')\times\r{id}}
   && {X'}^{2d}\times_{{X'}^{(2d)}}X^{(d)} \ar[d]^-{\r{pr}} \\
   \c{Y}^{2d}_{m,j} \ar@{^(->}[r]^-{\b{u}_{m,j}}
   \ar@/_3pc/[rrrr]_-{\b{d}_{m,j}}
   & \c{Y}_m^{2d} \ar[r]^-{\b{c}_m}
   \ar@/_1.5pc/[rrr]_-{\b{d}_m}
   & \Cohp^{2d}_0\times_{{X'}^{(2d)}}X^{(d)} \ar[rr]^-{\r{div}^{2d}\times\r{id}}
   && X^{(d)}   }
 \end{align*}
Similarly, we define
$\Ytc^{2d}_{m,j}=\c{Y}^{2d}_{m,j}\times_{\c{Y}^{2d}_m}\Ytc^{2d}_m$
which is a closed substack of $\Ytc^{2d}_m$ and an open substack of
$\Yt^{2d}_{m,j}$. We have restriction of morphisms
$\widetilde{\b{u}}^{\circ}_{m,j}$, $\b{u}^{\circ}_{m,j}$,
$\b{p}^{\circ}_{m,j}$, $\b{p}^{\circ}_m$,
$\widetilde{\b{d}}^{\circ}_{m,j}$, $\b{d}^{\circ}_{m,j}$,
$\widetilde{\b{d}}^{\circ}_m$ and $\b{d}^{\circ}_m$. But now the
restriction of $\widetilde{\b{c}}_m$ (resp. $\b{c}_m$) to
$\Ytc^{2d}_m$ (resp. $\Yc^{2d}_m$) will factor through
$\Flp^{2d}_{0,\leq m}\times_{\Cohp^{2d}_0}\Coh^d_0$ (resp.
$\Coh^d_{0,\leq m}$), and we define $\widetilde{\b{c}}^{\circ}_m$
(resp. $\b{c}^{\circ}_m$) to be the restriction but with this new
target. Let
$\widetilde{\b{c}}^{\circ}_{m,j}=\widetilde{\b{c}}^{\circ}_m\circ\widetilde{\b{u}}^{\circ}_{m,j}$.
We will have a similar diagram as above which we omit. In
particular, the stack $\Ytc^{2d}_{m,m}$ is naturally identified with
$\Zm^d_m\times_{\Coh^d_0}\(\Coh^d_0\times_{\Cohp^{2d}_0}\Flp^{2d}_0\)$.

As for $\Qp^0_m$, we define in the same way the evaluation map
$\r{ev}^{2d}:\Y^{2d}_m\ra\GA$, and $\AS^{2d}={\r{ev}^{2d}}^*\AS$,
$\ASt^{2d}=\b{p}_m^*\AS^{2d}$ which are rank-one local systems
\cf~\ref{num:steps-laumon}. The $*$-restrictions of $\ASt^{2d}$ to
$\Yt^{2d}_{m,j}$ and $\Ytc^{2d}_{m,j}$ are still denoted by
$\ASt^{2d}$. Then we have

\begin{proposition}\label{pro:red-red}
The morphisms $\widetilde{\b{d}}^{\circ}_m$ is of relative dimension
$\leq q_m+(m-1)d$, and the natural map of the highest cohomological
sheaves
 \begin{align*}
 \R^{2(q_m+(m-1)d)}\widetilde{\b{d}}^{\circ}_{m!}\ASt^{2d}\lra
 \R^{2(q_m+(m-1)d)}\widetilde{\b{d}}^{\circ}_{m,m!}\EL
 \end{align*}
is an isomorphism.
\end{proposition}
\begin{proof}
The estimation of the relative dimension will be accomplished in
Lemma \ref{lem:red-dim1}. For the second assertion, we inductively
prove that the natural map
 \begin{align*}
 \R^{2(q_m+(m-1)d)}\widetilde{\b{d}}^{\circ}_{m,j!}\ASt^{2d}\lra
 \R^{2(q_m+(m-1)d)}\widetilde{\b{d}}^{\circ}_{m,j+1!}\ASt^{2d}
 \end{align*}
is an isomorphism for $j=0,...,m-1$. Consider the commutative
diagram
 \begin{align*}
 \xymatrix{
   \Ytc^{2d}_{m,j+1} \ar[d] \ar[dr]^{\kappa^+} \ar@{^(->}[r]
   & \Ytc^{2d}_{m,j} \ar[d]^-{\kappa^{\circ}}
   \ar@{^(->}[r] & \Yt^{2d}_{m,j} \ar[dl]^{\kappa} \\
   \Yt^{2d}_{m-j,1} \ar@{^(->}[r] & \Yt^{2d}_{m-j}   }
 \end{align*}
where the square is Cartesian. The morphism $\kappa$ is a
generalized affine fibration of rank $q_m-q_{m-j}+jd$ and hence the
natural map
 \begin{align*}
 \R^{2(q_m-q_{m-j}+jd)}\kappa^{\circ}_{\;!}\ASt^{2d}\lra\R^{2(q_m-q_{m-j}+jd)}\kappa_!\ASt^{2d}
 \end{align*}
is an isomorphism over the image of $\kappa^{\circ}$. Applying Lemma
\ref{lem:res-key} to the coherent sheaf
$\s{N}/\s{M}_j\otimes\Omega_X^{-m+j+1}\otimes\s{L}_{\mu}^{-j+1}$, we
see that $\R^{2(q_m-q_{m-j}+jd)}\kappa_!\ASt^{2d}$ is supported on
the closed substack $\Yt^{2d}_{m-j,1}$. Hence, the natural map
 \begin{align*}
 \R^{2(q_m-q_{m-j}+jd)}\kappa^{\circ}_{\;!}\ASt^{2d}\lra\R^{2(q_m-q_{m-j}+jd)}\kappa^+_{\;!}\ASt^{2d}
 \end{align*}
is an isomorphism. To conclude the desired isomorphism at the
beginning of the proof, we only need to apply Lemma
\ref{lem:red-dim2} for $\widetilde{\b{d}}_{m-j}$. Now the second
assertion immediately follows since
$\ASt^{2d}|_{\Ytc^{2d}_{m,m}}=\EL$ by Lemma
\ref{lem:res-key}.\\
\end{proof}

\begin{corollary}\label{cor:red-coh}
Let $E$ and $L$ be as in Proposition \ref{pro:res-res}, then we have
a canonical isomorphism
 \begin{align*}
 \f{d}_{m!}\nu_m^*\(\W_{E,m}^{2d}\boxtimes\f{d}^*\sf{A}_L\)[-2d]\overset{\sim}{\lra}
 \(\R^{-2d}\r{div}^d_{\leq m!}\mu_0^*\Lau_E^{2d}\)\otimes L^{(d)}
 \end{align*}
of constructible sheaves on $X^{(d)}$.
\end{corollary}
\begin{proof}
By the above proposition, we have canonical isomorphisms
 \begin{align*}
 &\R^{2(q_m+md-d)}\f{d}_{m!}\(\left.{\overline{\f{h}}_{\leftarrow}'}{}^*\ASb^0\right|_{\Ytc^{2d}_m}\otimes
 \left.\b{c}_m^*\(\Spr^{2d}_E\boxtimes
 L^{(d)}\)\right|_{\Ytc^{2d}_m}\)\\
 \overset{\sim}{\lra}&\R^{2(q_m+md-d)}
 \r{pr}_!\f{d}_{m!}\b{p}^{\circ}_{m!}\(\ASt^{2d}\otimes
 \left.\widetilde{\b{c}}_m^*\(\(\r{div}^{\times2d*}E^{\boxtimes\:2d}\)\boxtimes
 L^{(d)}\)\right|_{\Ytc^{2d}_m}\)\\
 \overset{\sim}{\lra}&\R^{2(q_m+md-d)}\r{pr}_!\widetilde{\b{d}}^{\circ}_{m!}\(\ASt^{2d}\otimes\left.\widetilde{\b{c}}_m^*
 \(\(\r{div}^{\times2d*}E^{\boxtimes\:2d}\)\boxtimes
 L^{(d)}\)\right|_{\Ytc^{2d}_m}\)\\
 \overset{\sim}{\lra}&\R^{2(q_m+md-d)}\r{pr}_!\widetilde{\b{d}}^{\circ}_{m,m!}\left.\widetilde{\b{c}}_{m,m}^{\quad*}
 \(\(\r{div}^{\times2d*}E^{\boxtimes\:2d}\)\boxtimes
 L^{(d)}\)\right|_{\Ytc^{2d}_{m,m}}.
 \end{align*}
But the morphism
$\widetilde{\b{d}}^{\circ}_{m,m}:\Ytc^{2d}_{m,m}\ra{X'}^{2d}\times_{{X'}^{(2d)}}X^{(d)}$
is the composition of
$\widetilde{\b{c}}^{\circ}_{m,m}:\Ytc^{2d}_{m,m}\ra\Flp^{2d}_{0,\leq
m}\times_{\Cohp^{2d}_0}\Coh^d_0$ and the natural morphism
$\widetilde{\r{div}}{}^{2d}_{\leq m}:\Flp^{2d}_{0,\leq
m}\times_{\Cohp^{2d}_0}\Coh^d_0\ra{X'}^{2d}\times_{{X'}^{(2d)}}X^{(d)}$,
in which the first one is smooth and surjective with connected fibre
of dimension $q_m+md$ and the second one is of relative dimension
$\leq-d$ by Lemma \ref{lem:direct-dimension}. Hence
 \begin{align*}
 &\R^{2(q_m+md-d)}\r{pr}_!\widetilde{\b{d}}^{\circ}_{m,m!}\left.\widetilde{\b{c}}_{m,m}^{\quad*}
 \(\(\r{div}^{\times2d*}E^{\boxtimes\:2d}\)\boxtimes
 L^{(d)}\)\right|_{\Ytc^{2d}_{m,m}}\\
 \overset{\sim}{\lra}&\(\R^{-2d}\r{div}^d_{\leq m!}\mu_0^*\Spr_E^{2d}\)\otimes L^{(d)}
 \end{align*}
where all these sheaves carry natural actions of $\f{S}_{2d}$ and
the isomorphisms are equivariant under these actions. Taking
$\f{S}_{2d}$ invariants and by Corollary \ref{cor:red-coh}, we get the desired isomorphism.\\
\end{proof}
\end{num}

\begin{num}\label{num:red-dimension}
{\em Stratifications-II.} We recall a stratification on
$\Ztp^d_m:=\Zp^d_m\times_{\Cohp^d_0}\Flp^d_0$ introduced in
\cite{Ly} Section 4.3. For $\lambda\in\Lambda^d_{m,\r{eff}}$, the
stack $\Ztp^d_m\times_{\Qbp^d_m}\Qbp^{\lambda}_m$ is stratified by
locally closed substack $\Ztp^e_m$ for $e\in\:^d_mJ_d$ (in the
notation of \cite{Ly}), where $e=\(e^j_i\)$ is an $d\times m$ matrix
with entries being non-negative integers such that $\sum_ie^j_i=1$
for $j=1,...,d$ and $\sum_je^j_i=\lambda_i$ for $i=1,...,m$. Put
${X'}^e=\prod_{i,j}{X'}^{e^j_i}$, then there is a natural morphism
$\Ztp^e_m\ra{X'}^e\times_{{X'}^{\lambda}}\Qbp^{\lambda}_m$ is an
affine fibration of rank $d(\lambda)$. Now it is easy to deduce the
following.

\begin{lem}\label{lem:red-dim1}
The morphisms $\widetilde{\b{d}}^{\circ}_m$ is of relative dimension
$\leq q_m+(m-1)d$.
\end{lem}
\begin{proof}
The restriction of $\widetilde{\b{d}}^{\circ}_m$ to the stratum
$\Ztp^e_m\times_{\Qbp^{2\lambda}_m}\Qmb^{\lambda}_m$ factors as
 \begin{align*}
 \Ztp^e_m\times_{\Qbp^{2\lambda}_m}\Qmb^{\lambda}_m\lra
 \({X'}^e\times_{{X'}^{2\lambda}}X^{\lambda}\)\times_{X^{\lambda}}\Qmb^{\lambda}_m\lra
 {X'}^e\times_{{X'}^{2\lambda}}X^{\lambda}\lra{X'}^{2d}\times_{{X'}^{(2d)}}X^{(d)}
 \end{align*}
where the first morphism is an affine fibration of rank
$d(2\lambda)$ and $\Qmb^{\lambda}_m\ra X^{\lambda}$ is a generalized
affine fibration of rank $q_m+(m-1)d-2d(\lambda)$. Hence the lemma
follows.\\
\end{proof}

Now we estimate the relative dimension of $\widetilde{\b{d}}_m$.
There is a stratification of $\Yt^{2d}_m$ by the degree of the
maximal torsion subsheaf $\s{F}$ of $\s{N}$. We define a locally
closed subscheme $\Yt^{2d,2d'}_m$ of $\Yt^{2d}_m$ for $0\leq d'\leq
d$ classifying the data $(\s{N},(\s{M}'_i),(r'_i))\in\Yt^{2d}_m$
plus a flag of subsheaves
$\s{M}'_m=\s{N}'_0\subset\s{N}'_1\subset\cdots\subset\s{N}'_{2d}=\mu^*\s{N}$
and an exact sequence
 \begin{align*}
 \xymatrix@C=0.5cm{
   0 \ar[r] & \s{F} \ar[rr] && \s{N} \ar[rr] && \s{V} \ar[r] & 0 }
 \end{align*}
of $\s{O}_X$-modules such that $\s{F}$ is a torsion sheaf of degree
$d'$ and $\s{V}$ is a vector bundle of rank $m$. Let
$\s{F}'_i=\s{N}'_i\cap\mu^*\s{F}$ for $i=0,...,2d$, then the
successive inclusions of torsion sheaves
$0=\s{F}'_0\subset\s{F}'_1\subset\cdots\subset\s{F}'_{2d}=\mu^*\s{F}$
define a point of $\Flp_0^{2d'}\times_{\Cohp^{2d'}_0}\Coh^{d'}_0$
and an element $c\in\f{C}^{2d'}_{2d}$, which is the set of subsets
of $\{1,...,2d\}$ of cardinality $2d'$. Hence $c$ determines a
locally closed subscheme $\Yt^c_m$ of $\Yt^{2d,2d'}_m$ and for each
$c$, a morphism
$\Yt^c_m\ra\Flp_0^{2d'}\times_{\Cohp^{2d'}_0}\Coh^{d'}_0$ on one
hand. On the other hand, the quotient sheaves
$\s{V}'_i=\s{N}'_i/\mu^*\s{F}$ and the induced flag define a point
of $\Ytc^{2d-2d'}_m$. They together define a morphism
$\Yt^c_m\ra\(\Flp_0^{2d'}\times_{\Cohp^{2d'}_0}\Coh^{d'}_0\)\times\Ytc^{2d-2d'}_m$
which is a generalized affine fibration of rank $md'$. By Lemma
\ref{lem:red-dim1},
$\Ytc^{2d-2d'}_m\ra{X'}^{2d-2d'}\times_{{X'}^{(2d-2d')}}X^{(d-d')}$
is of relative dimension $\leq q_m+(m-1)(d-d')$, and by Lemma
\ref{lem:direct-dimension},
$\Flp_0^{2d'}\times_{\Cohp^{2d'}_0}\Coh^{d'}_0\ra{X'}^{2d'}\times_{{X'}^{(2d')}}X^{(d')}$
is of relative dimension $\leq-d'$, which together imply the the
following.

\begin{lem}\label{lem:red-dim2}
The morphisms $\widetilde{\b{d}}_m$ is of relative dimension $\leq
q_m+(m-1)d$.\\
\end{lem}
\end{num}

\section{Direct image of Laumon's sheaf}
\label{sec:direct}

\begin{num}
{\em Stratifications-III.} A point of
$\Flp^{2d}_0\times_{\Cohp^{2d}_0}\Coh^d_0$ is given by a torsion
sheaf $\s{F}$ on $X$ of degree $d$ and a complete flag of torsion
sheaves
$0=\s{F}'_0\subset\s{F}'_1\subset\cdots\subset\s{F}'_{2d}=\mu^*\s{F}=\sigma^*\mu^*\s{F}$
of $\mu^*\s{F}$. Let $\s{F}'_{i,j}=\s{F}'_i\cap\sigma^*\s{F}'_j$,
then $\s{F}'_{j,i}=\sigma^*\s{F}'_{i,j}$ as subsheaves of
$\mu^*\s{F}$. As in \cite{Ly} Section 6.3, define $\s{F}''_{i,j}$
from the co-Cartesian square
 \begin{align*}
 \xymatrix{
   \s{F}'_{i-1,j} \ar@{^(->}[r] & \s{F}''_{i,j} \\
   \s{F}'_{i-1,j-1} \ar@{^(->}[u] \ar@{^(->}[r] & \s{F}'_{i,j-1} \ar@{^(->}[u] }
 \end{align*}
for $i,j=1,...,2d$, and $\s{G}_{i,j}=\s{F}'_{i,j}/\s{F}''_{i,j}$.
Let $t_{ij}=\deg\s{G}_{i,j}$. Then it is easy to see that
$t=(t_{ij})$ is an element in $\f{T}_{2d}$ \cf~\ref{num:asai-sym}.
In this way, $\Flp^{2d}_0\times_{\Cohp^{2d}_0}\Coh^d_0$ is
stratified by locally closed schemes $\Flp^t_0$ for
$t\in\f{T}_{2d}$. For each $t$, the natural morphism
$\Flp^t_0\ra\prod_{i<j}\Cohp_0^{t_{ij}}$ by sending
$(\s{F},(\s{F}'_i))$ to $(\s{G}_{i,j})_{i<j}$ is a generalized
fibration of rank $0$. Since $\Cohp_0^1\ra X'\ra X$ is of relative
dimension $-1$, we have the following lemma whose second assertion
follows from the same argument of \cite{Ly} Lemma 11.

\begin{lem}\label{lem:direct-dimension}
The morphism
 \begin{align*}
 \widetilde{\r{div}}{}^{2d}:=(\r{div}^{\times2d}\circ\f{q}')\times\r{div}^d:\Flp^{2d}_0
 \times_{\Cohp^{2d}_0}\Coh^d_0\lra{X'}^{2d}\times_{{X'}^{(2d)}}X^{(d)}
 \end{align*}
is of relative dimension $\leq-d$ and we have a canonical
isomorphism
 \begin{align*}
 \R^{-2d}\widetilde{\r{div}}{}^{2d}_{\;\;!}\EL\overset{\sim}{\lra}\f{n}_!\EL
 \end{align*}
where $\f{n}$ is defined in \ref{num:asai-sym}.\\
\end{lem}
\end{num}

\begin{num}
{\em Proof of Proposition \ref{pro:res-coh}.} By the above lemma, we
have an isomorphism
 \begin{align*}
 \(\R^{-2d}\r{div}{}^d_{\;!}\mu_0^*\Spr_E^{2d}\)\otimes
 L^{(d)}\overset{\sim}{\lra}\r{pr}_!\(\(E^{\boxtimes\:2d}\boxtimes\EL\)\otimes\f{n}_!\EL\)
 \end{align*}
which is $\f{S}_{2d}$-equivariant. Taking invariants on both side
and by Lemma \ref{lem:asai-sym} (3), we have
 \begin{align*}
\(\R^{-2d}\r{div}{}^d_{\;!}\mu_0^*\Lau_E^{2d}\)\otimes
 L^{(d)}\overset{\sim}{\lra}\(\r{As}(E)\otimes L\)^{(d)}.
 \end{align*}
By Corollary \ref{cor:red-coh}, we have
 \begin{align*}
 &\f{d}_{m!}\nu_m^*\(\W_{E,m}^{2d}\boxtimes\f{d}^*\sf{A}_L\)[-2d]\\
 \overset{\sim}{\lra}&\(\R^{-2d}\r{div}^d_{\leq m!}\mu_0^*\Lau_E^{2d}\)\otimes
 L^{(d)}\hra\(\R^{-2d}\r{div}{}^d_{\;!}\mu_0^*\Lau_E^{2d}\)\otimes L^{(d)}\\
 \overset{\sim}{\lra}&\(\r{As}(E)\otimes L\)^{(d)}.
 \end{align*}
Denote $\(\r{As}(E)\otimes L\)^{(d)}_m\subset\(\r{As}(E)\otimes
L\)^{(d)}$ the image of
$\f{d}_{m!}\nu_m^*\(\W_{E,m}^{2d}\boxtimes\f{d}^*\sf{A}_L\)[-2d]$.
Then we have a filtration $0=(\r{As}(E)\otimes
 L)^{(d)}_0\subset(\r{As}(E)\otimes
 L)^{(d)}_1\subset\cdots\subset(\r{As}(E)\otimes
 L)^{(d)}_m\subset\cdots$ of $(\r{As}(E)\otimes L)^{(d)}$ such that
$(\r{As}(E)\otimes L)^{(d)}=\bigcup_m(\r{As}(E)\otimes L)^{(d)}_m$.
By Corollary \ref{cor:res-grade}, the filtration becomes stable
when $m\geq n$. Hence the proposition is proved.\\

\begin{remark}
In fact, it is not difficult to see that the inclusion
 \begin{align*}
 &\f{d}_{m!}\nu_m^*\(\W_{E,m}^{2d}\boxtimes\f{d}^*\sf{A}_L\)[-2d]\simeq(\r{As}(E)\otimes
 L)^{(d)}_m\\
 \hra&(\r{As}(E)\otimes L)^{(d)}_{m'}\simeq\f{d}_{m'!}\nu_{m'}^*\(\W_{E,m'}^{2d}\boxtimes\f{d}^*\sf{A}_L\)[-2d]
 \end{align*}
for $m\leq m'$ is compatible with the canonical filtration in
Corollary \ref{cor:res-grade}.\\
\end{remark}
\end{num}

{\em Acknowledgements.} The author would like to thank Weizhe Zheng
for helpful discussions.


\begin{thebibliography}{99}

\bibitem{BBD}Beilinson, A. A.; Bernstein, J.; Deligne, P.: Faisceaux
pervers, {\em Analyse et topologie sur les espaces singuliers (I)},
Ast\'{e}risque \textbf{100}, Soci\'{e}t\'{e} Math\'{e}matique de
France, Paris 1982, 5-171.

\bibitem{BG}Braverman, A.; Gaitsgory, D.: Geometric Eisenstein
series, Invent. Math. \textbf{150} (2002), 287-384.

\bibitem{Dr}Drinfeld, V.: Two-dimensional $\ell$-adic
representations of the fundamental group of a curve over a finite
field and automorphic forms on $\mathrm{GL}(2)$, Amer. J. Math.
\textbf{105} (1983), 85-114.

\bibitem{GJR}Gelbart, S.; Jacquet, H.; Rogawski, J.: Generic
representations for the unitary group in three variables, Israel J.
Math. \textbf{126} (2001), 173-237.

\bibitem{Fl1}Flicker, Y. Z.: Twisted tensors and Euler products,
Bull. Soc. math. France \textbf{116} (1988), 295-313

\bibitem{Fl2}Flicker, Y. Z.: On distinguished representations, J.
reine angew. Math. \textbf{418} (1991), 139-172.

\bibitem{FGKV}Frenkel, E.; Gaitsgory, D.; Kazhdan, D.; Vilonen, K.:
Geometric realization of Whittaker functions and the Langlands
conjecture, J. Amer. Math. Soc. \textbf{11} (1998), 451-484.

\bibitem{FGV1}Frenkel, E.; Gaitsgory, D.; Vilonen, K.: Whittaker
patterns in the geometry of moduli spaces of bundles on curves, Ann.
of Math. \textbf{153} (2001), 699-748.

\bibitem{FGV2}Frenkel, E.; Gaitsgory, D.; Vilonen, K.: On the
geometric Langlands conjecture, J. Amer. Math. Soc. \textbf{15}
(2002), 367-417.

\bibitem{GGP}Gan, W. T.; Gross, B. H.; Prasad, D.: Symplectic local root
numbers, central critical $L$-values, and restriction problems in
the representation theory of classical groups, Preprint
\href{http://www.math.harvard.edu/~gross/preprints/}
{http://www.math.harvard.edu/~gross/preprints/}, to appear in
Ast\'{e}risque.

\bibitem{LO}Laszlo, Y.; Olsson, M.: The six operations for sheaves
on Artin stacks I: finite coefficients \& II: adic coefficients,
Inst. Hautes \'{E}tudes Sci. Publ. Math. \textbf{107} (2008),
109-168 \& 169-210.

\bibitem{La1}Laumon, G.: Correspondance de Langlands
g\'{e}om\'{e}trique pour les corps de fonctions, Duke Math. J.
\textbf{54} (1987), 309-359.

\bibitem{La2}Laumon, G.: Faisceaux automorphes pour $\mathrm{GL}(n)$: la premi\`{e}re construction de
Drinfeld, Preprint
\href{http://arxiv.org/abs/alg-geom/9511004}{arXiv:alg-geom/9511004v1}
(1995).

\bibitem{LMB}Laumon, G.; Moret-Bailly, L.: {\em Champs
alg\'{e}briques}, Ergeb. Math. Grenzgeb. (3) \textbf{39}, Springer,
Berlin, 2000.

\bibitem{Ly}Lysenko, S.: Local geometrized Rankin-Selberg method for
$\mathrm{GL}(n)$, Duke Math. J. \textbf{111} (2002), 451-493.

\bibitem{Ng}Ng\^{o}, B. C.: Preuve d'une conjecture de
Frenkel-Gaitsgory-Kazhdan-Vilonen pour les groupes lin\'{e}aires
g\'{e}n\'{e}raux, Isreal J. Math. \textbf{120} (2000), 259-270

\bibitem{NP}Ng\^{o}, B. C.; Polo, P.: R\'{e}solutions de Demazure
affines et formule de Casselman-Shalika g\'{e}om\'{e}trique, J. Alg.
Geom. \textbf{10} (2001), 515-547.

\bibitem{Pr}Prasad, D.: Invariant forms for representations of
$\mathrm{GL}_2$ over a local field, Amer. J. Math. \textbf{114}
(1992), 1317-1363.

\end{thebibliography}
\end{document}